\documentclass[leqno,a4paper,10pt]{amsart}

\usepackage{amsmath, amsfonts, amsthm, amssymb, dsfont}

\usepackage{times}
\usepackage{microtype}

\usepackage{tikz}
\usetikzlibrary{arrows}
\RequirePackage[all, knot]{xy}
\xyoption{arc}
\usepackage{tikz-cd}
\usepackage{mathrsfs}
\usepackage{enumerate, xspace}
\usepackage{graphicx}
\usepackage{color}
\usepackage[marginpar]{todo}
\usepackage{enumitem}

\usepackage{mathtools}
\usepackage{dsfont}
\usepackage{physics}
\usepackage{latexsym}
\usepackage{stmaryrd}
\usepackage{eqnarray}

\usepackage[all, knot]{xy}
\xyoption{arc}

\usepackage{pgf,tikz}
\usepackage{mathrsfs}
\usetikzlibrary{arrows}

\usepackage{verbatim}

\usepackage{caption}
\usepackage{subcaption}

\usepackage[
	pdfauthor={Xabier Legaspi and Markus Steenbock},
	pdftitle={Uniform growth in small cancellation groups Version 2},
	pdfkeywords={growth, hyperbolic groups, hyperbolic space, acylindrical actions, small cancellation theory},
	pdftex, 
	breaklinks
]{hyperref}

\usepackage{cite}

\hypersetup{
	colorlinks=true,%
	citecolor=black,%
	filecolor=black,%
	linkcolor=black,%
	urlcolor=black,%
}

\usepackage{bm}

\theoremstyle{plain}

\newtheorem{thm}{Theorem}[section]

\newtheorem{prop}[thm]{Proposition}
\newtheorem{lem}[thm]{Lemma}		

\newtheorem{cor}[thm]{Corollary}

\newtheorem{cla}[thm]{Claim}

\theoremstyle{definition}		

\newtheorem{df}[thm]{Definition}
\newtheorem{ex}[thm]{Example}	
				
\newtheorem{rem}[thm]{Remark}

\newtheorem{quest}[thm]{Question}	

\theoremstyle{remark}

\DeclareMathOperator{\diam}{diam}

\DeclareMathOperator{\length}{length}

\newcommand{\cX}{\dot{X}}
\newcommand{\qX}{\overline{X}}
\newcommand{\qG}{\overline{G}}

\newcommand{\q}[1]{\overline{#1}}

\newcommand{\stab}[1]{\operatorname{Stab}(#1)}

\renewcommand{\leq}{\leqslant}
\renewcommand{\geq}{\geqslant}

\newcommand{\N}{\mathbf{N}}
\newcommand{\Z}{\mathbf{Z}}
\newcommand{\R}{\mathbf{R}}

\newcommand{\F}{\mathbf{F}}

\newcommand{\scrQ}{\mathscr{Q}}

\newcommand{\scrS}{\mathscr{S}}
\newcommand{\scrY}{\mathscr{Y}}

\newcommand{\scrV}{\mathscr{V}}

\let\le\leqslant
\let\ge\geqslant 
\let\leq\leqslant   
\let\geq\geqslant

\DeclareMathOperator{\Aut}{Aut}

\newcommand{\zinterval}[2]{
	\mathchoice%
	{\left \llbracket #1, \cdots, #2 \right \rrbracket}%
	{\llbracket #1, #2 \rrbracket}%
	{\llbracket #1, #2 \rrbracket}%
	{\llbracket #1, #2 \rrbracket}%
}

\newcommand{\group}[1]{
	\mathchoice%
	{\left\langle #1\right\rangle}%
	{\langle #1\rangle}%
	{\langle #1\rangle}%
	{\langle #1\rangle}%
}

\newcommand{\lnormal}{%
	\langle\!\langle%
}
\newcommand{\rnormal}{%
	\rangle\!\rangle%
}

\newcommand{\normal}[1]{%
	\lnormal #1 \rnormal
}

\newcommand{\sdirect}[3]{%
	#1 \rtimes_{#2} #3
}

\DeclarePairedDelimiterX\setc[2]{\{}{\}}{\,#1 \;\delimsize\colon\; #2\,}

\newcommand{\stlen}[1]{
	\left\| #1\right\|^{\infty}
}

\newcommand{\tlen}[1]{
	\left\| #1\right\|
}

\newcommand{\onto}{\twoheadrightarrow}

\newcommand{\into}{\hookrightarrow}

\title{Uniform growth in small cancellation groups}
\author{Xabier Legaspi}
\address{ICMAT, CSIC, 28049 Madrid, Spain\\ \&\\
 IRMAR, Universit\'e de Rennes 1, 35000 Rennes, France}
\author{Markus Steenbock}
\address{Fakult\"at f\"ur Mathematik, Universit\"at Wien,  1090 Wien, Austria}
\email{markus.steenbock@univie.ac.at}


\begin{document}
\let\thefootnote\relax\footnote{This is the author-accepted version of the original article published in J. London Math. Soc. 113 no.4  (2026), e70526, doi: \href{http://dx.doi.org/10.1112/jlms.70526}{10.1112/jlms.70526} under the terms of the \href{http://creativecommons.org/licenses/by/4.0/}{Creative Commons Attribution License}, which permits use, distribution and reproduction in any medium, provided the original work is properly cited.}
\begin{abstract}
An open question asks whether every group acting acylindrically on a hyperbolic space has uniform exponential growth. We prove that the class of groups of uniform uniform exponential growth acting acylindrically on a hyperbolic space is closed under taking certain geometric small cancellation quotients. There are two consequences: firstly, there is a finitely generated acylindrically hyperbolic group that has uniform exponential growth but has arbitrarily large torsion balls. Secondly,  the uniform uniform exponential growth rate of a classical $C''(\lambda)$-small cancellation group, for sufficiently small $\lambda$,  is bounded from below by a universal positive constant. We give a similar result for uniform entropy-cardinality estimates. This yields an explicit upper bound on the number of isomorphism classes of marked $\delta$-hyperbolic $C''(\lambda)$-small cancellation groups of uniformly bounded entropy in terms of $\delta$ and the entropy bound. 
\end{abstract}

\maketitle


\section{Introduction}
\label{sec:introduction}

Let $G$ be a group with finite symmetric generating set $U$. The $n$-th product $U^n$ is the subset of elements $g=u_1\cdot ... \cdot u_n\in G$ such that $u_1,\cdots,u_n\in U$. We study the \emph{exponential growth rate} 
$$\omega(U):=\limsup_{n \to \infty} \frac{1}{n}\log  |U^n|$$
where $|U^n|$ denotes the cardinality of the subset $U^n\subset G$. The exponential growth rate of a generating set is sometimes also referred to as of the  \emph{entropy} of the group with respect to the generating set.  
Let $\xi>0$. The group $G$ has \emph{$\xi$-uniform exponential growth} if for every finite symmetric generating set $U$ of $G$, we have $\omega(U)>\xi$. A group has \emph{$\xi$-uniform uniform exponential growth} if every finitely generated subgroup is either virtually nilpotent or has $\xi$-uniform exponential growth. A group $G$ has \emph{uniform exponential growth} or \emph{uniform uniform exponential growth} if there is $\xi>0$ such that $G$ has $\xi$-uniform exponential growth, or $\xi$-uniform uniform exponential growth, respectively.

Uniform exponential growth is particularly well-studied in groups of non-positive curvature. Indeed, groups of uniform uniform exponential growth include hyperbolic groups \cite{koubi_croissance_1998,arzhantseva_lower_2006,breuillard_joint_2021,besson_curvaturefree_2020}, free products of countable families of groups with $\xi$-uniform uniform exponential growth (folklore), groups hyperbolic relative to groups with $\xi$-uniform uniform exponential growth \cite{cui_lower_2022,kropholler_extensions_2020} as well as mapping class groups \cite{anderson_uniformly_2007,mangahas_uniform_2010}. Many properties of the groups in this list can be studied in the general framework of groups that act on hyperbolic spaces. Indeed, the groups in this list admit non-elementary acylindrical actions on hyperbolic spaces in the sense of  \cite{sela_acylindrical_1997,bowditch_tight_2008,dahmani_hyperbolically_2017}.  Uniform exponential growth is known for a wider class of groups, including, for example,  $Out(F_n)$, cocompactly special cubulated CAT(0) groups and generalisations of the latter, or the automorphism groups of one-ended hyperbolic groups  \cite{eskin_uniform_2005,anderson_uniformly_2007,gupta_groups_2023,abbott_hierarchically_2019,zalloum_effective_2023,wan_uniform_2025,kropholler_extensions_2020}. Some of these results actually include uniform uniform exponential growth for the groups in question, for instance, for certain groups acting on CAT(0) cubical complexes \cite{gupta_groups_2023}, for certain hierarchically hyperbolic groups \cite{wan_uniform_2025}, and  for the automorphism groups of one-ended hyperbolic groups \cite{kropholler_extensions_2020}. 
 It is open whether every group acting acylindrically on a hyperbolic space has uniform exponential growth. In fact, a positive answer would imply that there is a uniform bound on the exponential growth rates of all hyperbolic groups, independent of the hyperbolicity constant \cite{minasyan_acylindrically_2019}, see Example \ref{IE:osin} below for more information. This is a long-standing open question, see \cite[Question~2.1]{osin_algebraic_2004}\cite[Section~14, Question~2]{breuillard_joint_2021}. 

\subsection{Growth of small cancellation quotients.}

We prove that the class of groups of uniform uniform exponential growth acting acylindrically on a hyperbolic space is closed under taking $(\mu,\lambda)$-small cancellation quotients in the sense of \cite[Definition~6.25]{dahmani_hyperbolically_2017}. More precisely, this holds for sufficiently small  
 $\lambda$ and sufficiently large $\mu$ respectively. We refer to the $(\mu,\lambda)$-small cancellation as of $C''(\lambda,\mu)$-small cancellation.  
 
 We discuss this setting in more detail below. Before that we state the following result, which captures the essence of our main theorem.
\begin{thm}\label{IT:essence}
	Let $G$ be a group acting acylindrically and non-elementarily on a hyperbolic geodesic metric space. 
	Then there are $\lambda>0$ and $\mu>0$ such that the following are equivalent.
	\begin{enumerate}[label=(\roman*)]
		\item $G$ has uniform uniform exponential growth.
		\item Every $C''(\lambda,\mu)$-small cancellation quotient of $G$ has uniform uniform exponential growth.
		\item There is a $C''(\lambda,\mu)$-small cancellation quotient of $G$ that has uniform uniform exponential growth.
	\end{enumerate}
\end{thm}
The precise result is \autoref{IT:main} below. It shows that the respective growth rates do not depend on the acylindricity parameters. Before stating it, we fix some notation.

 Let $\delta>0$. Let $G$ be a group acting by isometries on a $\delta$-hyperbolic metric space $X$. For the purpose of this introduction, we may assume that hyperbolic spaces are geodesic. However, our results hold in the setting of metric length spaces, see Remark \ref{R:geodesic} below.  
 
\textit{Acylindricity}. Let $\kappa$, $N>0$. The action of $G$ on $X$ is $(\kappa,N)$-\emph{acylindrical} if for every pair of points $x, y \in X$ of distance at least $\kappa$, the number of elements $u \in G$ moving each of the points $x,y$ at distance at most $100\delta$ is bounded above by $N$.

\textit{Small cancellation}. Let ${R}\subset G$ be a set of loxodromic isometries, the \emph{relators}, that is closed under conjugation and inversion. Every relator $r\in R$ stabilises its quasi-convex axis $Y_r\subset X$. Let ${T}({R})$ be the minimal translation length of the elements of ${R}$. The set ${R}$ satisfies the \emph{$C''(\lambda,\mu)$-condition} (see Definition \ref{D:small-cancellation} below) if
\begin{align*}
{T}({R})>\mu\delta \hbox{ and } \diam\left(Y_{r_1}^{+\delta}\cap Y_{r_2}^{+\delta}\right)< \lambda {T}({R})\hbox{, unless $r_1=r_2^{\pm 1}$}.
\end{align*}
 In this case, the quotient by the normal closure of ${R}$ is a \emph{$C''(\lambda,\mu)$-small cancellation quotient} of $G$.
 
\begin{rem}\label{IR:smallcancellation}
The $C''(\lambda,\mu)$-condition is sensitive to the choice of $X$ and the action of $G$. 
If $r_1$ and $r_2 \in R$ are in the same maximal loxodromic subgroup, then $ r_1= r_2^{\pm 1}$. In addition, for every $r\in R$, the cyclic subgroup $\langle r \rangle$  is normal in the maximal loxodromic subgroup containing $r$. 
\end{rem}

\begin{rem} The $C''(\lambda,\mu)$-condition extends the classical $C''(\lambda)$-condition for quotients of free groups, which states that the word length of a piece in a relator word is shorter than the $\lambda$-fraction of the word length of a shortest relator.  Note that the classical $C'(\lambda)$-condition states that the word length of a piece in a relator word is shorter than the $\lambda$-fraction of the word length of any relator that contains this piece. Hence, the $C''(\lambda)$-condition is stronger than the $C'(\lambda)$-condition. We come back to these facts in Section~\ref{S:classical small cancellation} below.
\end{rem}

\textit{Main theorem}. 
The main theorem of this article is the following.
\begin{thm}[\autoref{IT:main2-body} \& \autoref{IT:main3-body}]
	\label{IT:main}
	For every $N >0$, there is $\lambda_0>0$ with the following property. For every $\delta>0$ and $\kappa\ge \delta$ there is $\mu_0> \kappa/\delta $ such that for all $0<\lambda\leqslant \lambda_0$ and all $\mu\geqslant \mu_0$ the following holds. 
	
	Let $\xi>0$. Let $G$ be a group acting $(\kappa,N)$-acylindrically on a $\delta$-hyperbolic space. 
	\begin{enumerate}[label=(\roman*)]
		\item If $G$ has $\xi$-uniform uniform exponential growth, then every $C''(\lambda,\mu)$-small cancellation quotient of $G$ has $\xi'$-uniform uniform exponential growth, where $\xi'\geqslant \min\{\frac{\xi}{10^8},\frac{1}{10^5}\log 2\}$.
		\item If there is a $C''(\lambda,\mu)$-small cancellation quotient of $G$ of $\xi$-uniform uniform exponential growth, then $G$ has $\xi'$-uniform uniform exponential growth, where  $\xi'\geqslant \min\{\frac{\xi}{10^8},\frac{1}{10^5}\log 2\}$.
	\end{enumerate}
Moreover, every $C''(\lambda,\mu)$-small cancellation quotient of $G$ acts acylindrically  on a hyperbolic space. 
\end{thm}

 The growth rate $\xi'$ of such a $C''(\lambda,\mu)$-small cancellation quotient is independent of the acylindricity and hyperbolicity constants, see Theorem \ref{IT:main}(i). 
 This contrasts the usual bounds on the growth rate: if $G$ is hyperbolic, the known lower bounds on the exponential growth rate of a symmetric generating set depend on $\delta$. See \cite[Theorem 1.14]{breuillard_joint_2021} and \cite[Corollary 1.4]{besson_curvaturefree_2020}. The known lower bounds on the uniform uniform exponential growth rate depend  on the size of the ball of radius $10\delta$ \cite{koubi_croissance_1998,arzhantseva_lower_2006} or on the covering constant of balls of radius $10\delta$ \cite[Theorem 1.13]{breuillard_joint_2021}.

\begin{rem} The dependence of the small cancellation parameter $\lambda$ and $\mu$ on $N$ can be slightly relaxed. In fact $\lambda$ and $\mu$ depend on the maximal cardinality of the finite normal subgroups of loxodromic subgroups. In particular, there is a universal constant $\lambda$ so that Theorem \ref{IT:main} holds if all loxodromic subgroups of $G$ are cyclic.
\end{rem}

\begin{rem} The dependence of the relator length $\mu$ on $N$, $\kappa$ and $\delta$ is not a strong condition. Indeed, the diameter of the intersection of the axis of two independent loxodromic elements is controlled in terms of $\kappa$, $N$, $\delta$ and the injectivity radius. Thus, to prove that a set of relators satisfies the  $C''(\lambda,\mu)$-condition, one usually considers relators of sufficiently large translation length compared to $\kappa$, $N$ and $\delta$ anyway. 
\end{rem}

\begin{rem} We could alter the definition of uniform uniform exponential growth, and say that $G$ has uniform uniform exponential growth if there is $\xi>0$ such that every finitely generated subgroup of exponential growth has $\xi$-uniform exponential growth. This variant of uniform uniform exponential growth is sometimes considered in the literature, see for instance \cite{cui_lower_2022,kropholler_extensions_2020}. Theorems \ref{IT:essence} and \ref{IT:main} then remain valid. In fact, if $G$ acts acylindrically on a hyperbolic space, then every non-elementary subgroup of $G$ contains a free subgroup, hence, it has exponential growth. On the other hand, virtually nilpotent groups have polynomial growth, hence, virtually nilpotent subgroups are elementary for every acylindrical action of $G$ on a hyperbolic space.  
\end{rem}

\subsection{Beyond short loxodromics.}
The standard strategy to show uniform exponential growth in a negative curvature setting applies if every finite symmetric generating set $U$  has the \emph{short loxodromic property}, that is, if every $n$-th power $U^n$ contains a loxodromic isometry, for some number $n$ that does not depend on the set $U$. See e.g. \cite{koubi_croissance_1998} or \cite{mangahas_uniform_2010}. By  a theorem of \cite{minasyan_acylindrically_2019}, this is not always the case: 
\begin{ex}\label{IE:osin}
Let $Q$ be a finitely generated group that is a quotient of every hyperbolic group and that admits an acylindrical action on a hyperbolic space \cite[Theorem 1.1]{minasyan_acylindrically_2019}. This group does not have the short loxodromic property \cite[Corollary 1.4]{minasyan_acylindrically_2019}. In addition, it is open whether $Q$ has uniform exponential growth. A positive answer would imply that there is a uniform bound on the exponential growth rates of all hyperbolic groups, independent of the hyperbolicity constant. Recall that this is a long-standing open question, see  \cite[Question~2.1]{osin_algebraic_2004}\cite[Section~14, Question~2]{breuillard_joint_2021}.
\end{ex}
Our main result, \autoref{IT:main}, does not make use of the short loxodromic property. In fact, there are families of hyperbolic groups with the following properties. 

\begin{ex}[Proposition \ref{P: groups-large-torsion-balls}]\label{IE:infinite} There are $\delta>0$, $\kappa \geq \delta$, $N>0$  and $\xi >0$ such that the following holds. For every $r>0$, there is a hyperbolic group $H_r=\langle S_r\rangle$ such that all elements in the ball $(S_r\cup S_r^{-1})^r$ are torsion elements. Moreover, $H_r$ acts $(\kappa,N)$-acylindrically and non-elementarily on a $\delta$-hyperbolic space and has $\xi$-uniform uniform exponential growth, for the uniform $\delta>0$, $\kappa \geq \delta$, $N>0$  and $\xi >0$. Thus, the uniform uniform exponential growth rate of the respective small cancellation quotients, see \autoref{IT:main} (i), does not depend on the cardinality of the torsion balls $(S_r\cup S_r^{-1})^r$. 
\end{ex}

In Proposition \ref{P: groups-large-torsion-balls} below, we view the existence of the groups of Example \ref{IE:infinite} as a consequence of  the fact that for sufficiently large odd exponent $n$, infinite Burnside groups all of whose elements have exponent $n$ have uniform uniform exponential growth, see \cite{coulon_product_2022}. 

In fact, we obtain finitely generated groups with the following properties. 
\begin{cor}\label{IC:common-quotients} There are finitely generated groups  without the short loxodromic property that have uniform exponential growth and that act acylindrically and non-elementarily on a hyperbolic space. 
\end{cor}
These groups, see Corollary \ref{IC:common-quotients}, are common quotients of the groups in Example \ref{IE:infinite}. We build them  as a $C''(\lambda,\mu)$-quotient so that Theorem \ref{IT:main} applies. 
It remains open whether the groups of \cite{minasyan_acylindrically_2019}, see Example \ref{IE:osin}, have uniform exponential growth.

\subsection{Classical small cancellation groups} \label{S:classical small cancellation}
We turn to groups given by a presentation under the the classical $C''(\lambda)$-small cancellation condition. We refer to a group that admits such a presentation as classical $C''(\lambda)$-small cancellation group. These are $C''(\lambda,\mu)$-small cancellation quotients of  free groups, where the parameter $\mu=0$. If $\lambda \leqslant 1/6$, a classical $C''(\lambda)$-small cancellation group is always finitely presented and hyperbolic. Thus it has uniform uniform exponential growth by \cite{gromov_hyperbolic_1987, koubi_croissance_1998}. Note that uniform uniform exponential growth is not explicitly stated in \cite{koubi_croissance_1998}, but follows from the results therein, see Remark \ref{R:proof_koubi} below. 
Alternatively, one can apply \cite{arzhantseva_lower_2006}. 
However, in the approach of these papers the uniform uniform exponential growth rate depends on the length of the relators. In other words, it depends on the hyperbolicity constant of the Cayley graph.

 The following is a consequence of \autoref{IT:main}.

\begin{cor}
	\label{IT:growth-small-cancellation}
	There are $\xi>0$ and $\lambda_0>0$ such that every classical $C''(\lambda)$-small cancellation group has $\xi$-uniform uniform exponential growth, for all $0<\lambda \leqslant \lambda_0$ .
\end{cor}

There is a generic class of classical $C''(1/6)$-small cancellation groups such that every $2$-generated subgroup is free  \cite{arzhantseva_generality_1996}. This immediately implies Corollary \ref{IT:growth-small-cancellation} for this generic class of classical $C''(1/6)$-small cancellation groups \cite[p. 194]{de_la_harpe_topics_2000}. 

\begin{quest} Is there a constant $\xi>0$ such that every $C''(1/6)$-small cancellation group has $\xi$-uniform uniform exponential growth? Is there a constant $\xi>0$ such that every $C'(1/6)$-small cancellation group has $\xi$-uniform uniform exponential growth? 
\end{quest}

\begin{rem} 
	The classical $C''(\lambda)$-small cancellation condition in Corollary \ref{IT:growth-small-cancellation} is reminiscent of our proof that uses geometric small cancellation theory. To this date, geometric small cancellation theory has not been developed under a  $C'(\lambda,\mu)$-small cancellation condition. We expect, however, that this is possible, and thus that our results hold for classical $C'(\lambda)$-small cancellation groups - finitely and infinitely presented.
\end{rem}

\subsection{Product set growth}
We say that a group has \emph{product set growth} (for its symmetric subsets) if there is $a>0$ such that for all finite symmetric sets $U$ that are not in a virtually nilpotent subgroup $\omega(U)\ge a \log(|U|)$. In this case, we say that $G$ has \emph{product set growth with constant $a$}. If $G$ acts acylindrically on a hyperbolic space and has the short loxodromic property then $G$ has product set growth for its symmetric generating sets \cite[Proposition 2.10]{fujiwara_rates_2021}, see Theorem \ref{thm:growth-trichotomy-intro} below. See also \cite{kerr_product_2021} for groups acting on quasi-trees, and \cite{delzant_product_2020,coulon_product_2022,wan_uniform_2025,fioravanti_product_set_growth_2024} for results on the growth of products of general subsets of $G$. 
 The following is a variant of Theorem~\ref{IT:main}. 
 
\begin{thm}[Remark \ref{R:product-set-growth}] \label{IT:product-set-growth} For every $N >0$, there is $\lambda>0$ such that the following holds. For every $\delta>0$ and $\kappa\ge \delta$ there is $\mu> \kappa/\delta $ with the following property. Let $G$ be a group acting $(\kappa,N)$-acylindrically on a $\delta$-hyperbolic space. 
	\begin{enumerate}[label=(\roman*)]
		\item If $G$ has product set growth, then every $C''(\lambda,\mu)$-small cancellation quotient of $G$ has product set growth. 
		\item If there is a $C''(\lambda,\mu)$-small cancellation quotient of $G$ that has product set growth, then $G$ has product set growth. 
	\end{enumerate}
\end{thm}

Again, the short loxodromic property is not used in the proof of Theorem \ref{IT:product-set-growth}. 

\begin{rem}\label{IR:product-set-growth} If all elementary subgroups of $G$ are cyclic and if $G$ has product set growth with constant $a$, then every $C''(\lambda,\mu)$-small cancellation quotient of $G$ of Theorem \ref{IT:product-set-growth}(i) has product set growth with a constant $a'=a'(a,N)$ that can be explicitly computed.
\end{rem}

Product set growth of hyperbolic groups is the starting point in Fujiwara-Sela's proof that the minimal growth rate in hyperbolic groups is attained \cite{fujiwara_rates_2020}. Indeed, product set growth allows to restrict to generating sets of bounded cardinality. Also, the minimal growth rate of every equationally Noetherian group that acts acylindrically on a hyperbolic space and that has the short loxodromic property is attained \cite{fujiwara_rates_2021}. In view of Theorem \ref{IT:product-set-growth}, it would be interesting to adapt this result for small cancellation quotients of such groups. 

 Product set growth has also been used in the context of finiteness results for marked $\delta$-hyperbolic groups. See \cite{cerochi_entropy_2021}. Here a marked group $(G,\Sigma)$ is $\delta$-hyperbolic if its Cayley graph with respect to the generating set $\Sigma$ is $\delta$-hyperbolic. In fact, the  isomorphism class of marked  $\delta$-hyperbolic groups with product set growth with a uniform constant $a>0$ and of uniformly bounded entropy is bounded by an explicit function in $\delta$, $a$ and the entropy bound \cite[Proof of Theorem 3.1]{cerochi_entropy_2021}. Thus Theorem \ref{IT:product-set-growth} and Remark \ref{IR:product-set-growth} applied to classical $C''(\lambda)$-small cancellation groups yields the following.
 
 \begin{cor}\label{IC:finiteness} There is $\lambda_0>0$ such that for all $0< \lambda \leqslant \lambda_0$ the following holds. Let $\delta,E>0$. The number of isomorphism classes of marked $\delta$-hyperbolic groups with an entropy bounded by $E$ and that admit a classical $C''(\lambda)$--small cancellation   presentation is bounded by an explicit number $M=M(\delta,E)$.  
 \end{cor}
 Note that the $C''(\lambda)$--small cancellation presentation, see Corollary \ref{IC:finiteness}, is not necessarily with respect to the fixed generating set.  In general, the number of isomorphism classes of marked torsion-free non-elementary $\delta$-hyperbolic groups of uniformly bounded entropy is bounded by a number that only depends on $\delta$ and the entropy bound \cite[Theorem 1.4]{besson_finiteness_2021}. It is not known whether this number can be explicitly computed in the case of torsion-free $\delta$-hyperbolic groups.

\subsection{Strategy of proof.} 
Let  $G$ act by isometries on $X$. The $\ell^{\infty}$-energy $L(U)$ of a finite subset $U\subset G$ is defined by 
$$
L(U)=\inf_{x\in X} \max_{u\in U} |ux-x|.
$$
If $U=\{g\}$, the $\ell^{\infty}$-energy coincides with the translation length of $g$. The following example explains why the $\ell^{\infty}$-energy is relevant in our context. 

\begin{ex}If $G$ is the fundamental group of a compact hyperbolic manifold, there exists a constant $\varepsilon>0$ -- the \emph{Margulis constant} -- such that if $U\subset G$ is a finite set with $L(U)<\varepsilon$, then the subgroup of $G$ generated by $U$ is virtually nilpotent.
\end{ex}

If $U\subset G$, we denote by $U^{-1}$ the set of the inverses of the elements of $U$. 

\begin{df}[cf. Definition \ref{df:reduced-subset}]\label{DI:reduced}
A subset $U\subset G$ is \emph{reduced at} $p\in X$ if $U\cap U^{-1}=\varnothing$ and for every pair of distinct $u_1, u_2 \in U\sqcup U^{-1}$, the Gromov product satisfies
	$$
	(u_1p, u_2p)_p<\frac{1}{2}\min\{|u_1p-p|,|u_2p-p|\}-250\delta.
	$$	
\end{df} 

\begin{rem}
	Roughly speaking, if a  set $U\subset G$ is reduced then  the orbit map from the free group generated by $U$ to $X$ is a quasi-isometric embedding.
\end{rem}

The following is a theorem of \cite{koubi_croissance_1998, arzhantseva_lower_2006,fujiwara_rates_2021}, see Proposition \ref{P:Pingpong_intro} below. 

\begin{thm}
	\label{thm:growth-trichotomy-intro}
	For every $N>0$, there is an integer $c>1$ with the following property. Let $\delta>0$ and $\kappa >0$. Let $G$ be a group acting $(\kappa,N)$-acylindrically on a $\delta$-hyperbolic space. Let $U\subset G$ be a finite symmetric subset. Then one of the following holds:
	\begin{enumerate}[label=(\roman*)]
		\item $L(U)\le 10^4\max\{\kappa,\delta\}$. 
		\item The subgroup $\group{U}$ is virtually cyclic and contains a loxodromic element.
		\item There is a reduced subset $S\subset U^c$ such that
		$$|S| \ge \max \{ 2, \frac{1}{c} |U| \}.$$
		Moreover, 
		$$\omega(U) \ge \frac{1}{2c} \log |U|.$$
	\end{enumerate}
\end{thm}
\begin{rem} The number $c$ in Theorem \ref{thm:growth-trichotomy-intro} only depends on $N$, and not on $\kappa$ as could perhaps be expected from \cite[Proposition 2.10]{fujiwara_rates_2021}. This eventually allows to show that the small cancellation parameter $\lambda$ of Theorem \ref{IT:main} only depends on $N$. 
\end{rem}

\begin{rem} 
	If the injectivity radius of the action of $G$ on $X$ is large, then every finite symmetric subset of $G$ satisfies either (ii) or (iii). In general this is however not the case.  
\end{rem}

To prove \autoref{IT:main} (i), we closely follow a strategy of \cite{coulon_product_2022} that estimates product set growth in Burnside groups. Let $\delta>0$ and $G$ act acylindrically on $X$.  Geometric small cancellation theory provides a universal constant $\q \delta>0$ such that every $C''(\lambda,\mu)$-quotient, for appropriately chosen $\lambda$ and $\mu$, acts acylindrically on a $\q \delta$-hyperbolic space. Let $\q U\subset \q G$ be a finite symmetric set that is not contained in an elliptic or virtually cyclic subgroup. If $L(\q U)>10^4 \q \delta$, then the exponential growth rate of $\q{U}$ is bounded below by a universal strictly positive constant.  Otherwise, we fix a pre-image $U$ of $\q{U}$ in $G$ of minimal energy. Such a pre-image may not have large energy $>10^4\delta$. Indeed, it may consist entirely of torsion-elements and thus have energy $<10^4\delta$. However, $U$ is not contained in any elliptic subgroup.  Thus some power of $U$ contains a loxodromic element, hence, for some exponent $n$, we have $L(U^{n})>10^4\delta$. We stress that the exponent $n$ depends on the set $U$.

Let us apply \autoref{thm:growth-trichotomy-intro} to $U^{n}$. Since $U$ is not contained in any virtually cyclic subgroup, we obtain a reduced subset $S$ in $U^{cn}$, which freely generates a free subgroup. Next, we adapt a counting argument of \cite{coulon_growth_2013,coulon_product_2022} for aperiodic words to prove that, for every $r\ge 1$, the proportion of elements in $S^r$ that contain a large part of a relator is small compared to $|S^r|$ (Proposition \ref{prop:growth shortening-free}). A priori, the argument of \cite{coulon_product_2022} is for ``strongly reduced subsets'' in the sense of \cite[Definition 3.1]{coulon_product_2022}. We adapt the set up of the counting argument accordingly, so that the argument carries over to reduced subsets in the sense of Definition \ref{DI:reduced}. A combination of {Greendlinger's Lemma} and {Fekete's Subadditive Lemma} then implies that the exponential growth rate of $\q{U}$ satisfies
$$\omega(\q U)\ge \frac{ \omega(U)}{2c},$$ see Lemma \ref{lem:small-energy} below.  
Recall that $c$ depends on $N$. 
To correct for this, we  increase the exponent $n$ as to make it grow with $N$. This yields Proposition \ref{P:Pingpong} below. Then the dependencies of $c$ and $n$ on $N$ cancel out in the final estimate, see Lemma \ref{lem:small-energy} below. The proof of \autoref{IT:main} (ii) is similar.

\subsection{Acknowledgements}
A previous version of this article is part of the doctoral thesis of the first author \cite{legaspijuanatey:tel-04251309}. 
We thank Rémi Coulon for encouraging us to do this work, and for useful discussions and comments on a previous version of this work, and Yago Antolín for a careful reading and comments on previous versions. We also thank Goulnara Arzhantseva, Chris Cashen and Ashot Minasyan for useful discussions. We are grateful to the referees for their careful reading and useful comments that helped us improve the overall quality of the article. 

Both authors were supported in parts by LabEx CARMIN, ANR-10-LABX-59-01 of the Institut Henri Poincar\'e (UAR 839 CNRS-Sorbonne Universit\'e) during the program \emph{Groups Acting on Fractals, Hyperbolicity and Self-Similarity}.
The first author was supported by \emph{Centre Henri Lebesgue} ANR-11-LABX-0020-01 and the grants  SEV-2015-0554-18-4 and PID2021-126254NB-I00 funded by MCIN/AEI/10.13039/501100011033. This research was funded in whole or in part by the Austrian Science Fund (FWF) [10.55776/P35079]. For open access purposes, the authors have applied a CC BY public copyright license to any author-accepted manuscript version arising from this submission.

\section{Hyperbolic geometry}

\label{sec:hyperbolic-geometry}

We collect some facts on hyperbolic geometry in the sense of \cite{gromov_hyperbolic_1987,coornaert_geometrie_1990,ghys_sur_1990}.

\subsection{Hyperbolicity}

\label{subsec:hyperbolicity}

Let $X$ be a metric length space. Given $x$, $x' \in X$ we write $|x-x'|$ for their distance and $[x,x']$ for a geodesic joining them (given that such a path exists). If $Y \subset X$ is a subset and $x \in X$, we write $d(x,Y) = \inf_{y\in Y} |x-y|$ to denote the \emph{distance from $x$ to $Y$}. 
 Given $\varepsilon \ge 0$, we let $Y^{+\varepsilon}=\{x \in X \mid d(x,Y) \le \varepsilon\}$ be the \emph{$\varepsilon$-neighbourhood of} $Y$. The \emph{Gromov product} of $x$, $y$, $z \in X$  is defined by
$$
	(x,y)_z = \frac 12 \{ |x-z| + |y-z| - |x-y| \}.
$$

\begin{df}
\label{df:hyperbolicity}
Let $\delta\ge 0$. The space $X$ is \emph{$\delta$-hyperbolic} if for every $x$, $y$, $z$ and $t\in X$, the \emph{four point inequality} holds, that is
$$
	(x,z)_t\ge \min \{ (x,y)_t,(y,z)_t \} - \delta.
$$
\end{df}

In this paper, all hyperbolic spaces are metric length spaces. We need the notion of metric length spaces in order to state a small cancellation lemma, Lemma \ref{lem:small-cancellation-theorem} below. Apart from that, we develop our theory in the language of metric geodesic spaces. In a metric length space there may not be a geodesic between two points, but there still is a $(1,\varepsilon)$-quasi-geodesic between any two points, for all $\varepsilon >0$. Up to increasing the number of $\delta$ if needed, our results thus hold for metric length spaces as well. In fact, we have the following. 

 \begin{rem}\label{R:geodesic} \cite[Lemma 5.34]{dahmani_hyperbolically_2017} implies that every $\delta$-hyperbolic metric length space $X$ injects into a $9\delta$-hyperbolic metric graph $X'$ of fixed uniform edge length, and this injection is a $(1,\delta)$-quasi-isometric embedding. That is, every point of $X'$ is at distance $\leq \delta$ from $X$ and $|x-y|_X\leqslant |x-y|_{X'}\leqslant |x-y|_X+\delta$, for all $x,y\in X$. The vertex set of $X'$ is the set $X$, and there is an edge of length $=\delta$ between $x,y\in X$ if and only if $|x-y|_X\leqslant \delta$. This implies that $X'$ is $9\delta$-hyperbolic, and by construction $X'$ is geodesic. Moreover, if there is an action of a group $G$ by isometries on $X$, it induces an action of $G$ by isometries on $X'$. If the action on $X$ is non-elementary, then so is the action on $X'$. If, in addition, the action of $G$ on $X$ is $(\kappa,N)$-acylindrical in the sense of Definition \ref{def: acylindricity} below, then the action of $G$ on $X'$ is $(\kappa',N')$-acylindrical, for some $\kappa'\geqslant\kappa$ and some $N'\geqslant N$. One can for instance choose $\kappa'=5000\kappa$ and $N'=320 N$. This follows by the properties of the construction and by the estimates given in the proof of \cite[Proposition 5.31]{dahmani_hyperbolically_2017}.
\end{rem}

From now on, we assume that the space $X$ is $\delta$-hyperbolic and geodesic.  
Note that $X$ is $\delta'$-hyperbolic for every $\delta' \ge \delta$. Without restriction  we also assume that $\delta >0$.

 Hyperbolicity has the following consequences.

\begin{lem}[\hspace{-0.15mm}{\cite[Lemmas~2.3~and~2.4]{delzant_product_2020}}]
	 
	\label{lem:projections-gromov-product}
	
	Let $x,y,z\in X$. Then
	$$
	(x,y)_z \le d(z,[x,y]) \le (x,y)_z+4\delta.
	$$
	
\end{lem}

\begin{lem}[\hspace{-0.15mm}{\cite[Lemma~2]{arzhantseva_lower_2006}}]
	
	\label{lem:intersection-four-points}
	
	Let $i\in \{1,2\}$. Let $x_i$, $y_i\in X$. Then
	$$|x_1-y_1|+|x_2-y_2|\le |x_1-x_2|+|y_1-y_2|+2\diam([x_1,y_1]^{+8\delta}\cap[x_2,y_2]^{+8\delta}).$$
	
\end{lem}

Let $k \ge 1$ and $l \ge 0$. Let $\gamma:[a,b]\to X$ be a rectifiable path with $a,b\in \R\cup\qty{-\infty,\infty}$. The path $\gamma$ is a \emph{$(k,l)$-quasi-geodesic} if for all $[a',b']\subset [a,b]$,
$$\length(\gamma[a',b'])\le k |\gamma(a')-\gamma(b')|+l.$$ 
Let $L\ge 0$. We say that $\gamma$ is a \emph{$L$-local $(k,l)$-quasi-geodesic} if any subpath of $\gamma$ of length at most $L$ is a $(k,l)$-quasi-geodesic. 

\begin{lem}[\hspace{-0.15mm}{\cite[Corollary~2.7]{coulon_geometry_2014}}]
\label{lem:quasi-line}
Let $\gamma\colon I\to X$ be a $10^3\delta$-local $(1,\delta)$-quasi-geodesic. Then:
\begin{enumerate}[label=(\roman*)]
	\item For every $t,t',s\in I$ such that $t\le s\le t'$, we have $(\gamma(t),\gamma(t'))_{\gamma(s)}\le 6\delta$.
	\item For every $x\in X$ and for every $y,y'\in \gamma$, we have $d(x,\gamma)\le (y,y')_x+9\delta$.
\end{enumerate}
\end{lem}

From now on let $G$ be a group acting by isometries on $X$. Let $x \in X$ be a point.
  
\subsection{Translation lengths.} 

To measure the displacement of an isometry $g$ on $X$ we define the \emph{translation length} and the  \emph{stable translation length} as
$$
	\tlen g = \inf_{x \in X} |gx-x|, 
	\quad \text{and} \quad
	\stlen g = \lim_{n \rightarrow + \infty} \frac 1n |g^nx-x|.
$$
The definition of $\stlen{g}$ does not depend on choice of $x$. These two lengths are related as follows, \cite[Chapitre~10, Proposition~6.4]{coornaert_geometrie_1990}.
\begin{equation}	
	\label{eqn:translation lengths}
	\stlen g \leq \tlen g \leq \stlen g + 16\delta.
\end{equation}
Note that $g$ is \emph{loxodromic} if, and only if, $\stlen g >0$, \cite[Ch.~10, Prop.~6.3]{coornaert_geometrie_1990}.

\subsection{Acylindricity}
\label{subsec:group-action}

We use the following definition of acylindricity of the action, see \cite[Proposition~5.31]{dahmani_hyperbolically_2017}. Recall that $\delta>0$.

\begin{df}
	\label{def: acylindricity}
	Let $\kappa$, $N>0$. The group $G$ acts \emph{$(\kappa,N)$-acylindrically on $X$} if the following holds: for every $x, y \in X$ with $|x-y|\ge \kappa$, the number of elements $u \in G$ satisfying $|ux-x| \le 100\delta$ and $|uy-y| \le 100\delta$ is bounded above by $N$.
\end{df}

 We write $\partial X$ for the boundary of $X$, see  \cite[Chapitre~2, Définition~1.1]{coornaert_geometrie_1990}.
We denote by $\partial G$ the set of all accumulation points of an orbit $G \cdot x$ in $\partial X$.
This set does not depend on the choice of the point $x$. A subgroup $H$ of $G$ is \emph{non-elementary} if $\partial H$ contains at least $3$ points, or equivalently if $\partial H$ is infinite. Otherwise $H$ is called \emph{elementary}. Under the assumption that the action of $G$ on $X$ is acylindrical, an elementary subgroup is either elliptic, that is, $\partial H$ is empty, or loxodromic, that is, $\partial H$ contains exactly $2$ points. 

Let $H\le G$ be a loxodromic subgroup with limit set $\partial H=\{ \xi,\eta \}$. 
The \emph{maximal loxodromic subgroup containing $H$} is the stabiliser of the set $\partial H$. For a loxodromic element $g\in G$, we denote by $E(g)$ the \emph{maximal loxodromic subgroup containing $g$}. 
\subsection{Loxodromic subgroups.}

From now on we assume that $\kappa\geqslant \delta$ and that the action of $G$ on $X$ is $(\kappa,N)$-acylindrical.

\begin{lem}[\hspace{-0.15mm}{\cite[Lemma~6.5]{dahmani_hyperbolically_2017}}]
	\label{lem:virtually cyclic elementary closure}   
	Let $g\in G$ be a loxodromic element. Then $E(g)$ is virtually cyclic.
\end{lem}

Let $H$ be a loxodromic subgroup. 
The subgroup $H^+\le G$ fixing pointwise $\partial H$ is a subgroup of $H$ of index at most $2$. The next corollary is a consequence of Lemma \ref{lem:virtually cyclic elementary closure} and \cite[Lemma~4.1]{wall_poincare_1967}.

\begin{cor}
	
	\label{lem:loxodromic-subgroups-f-by-z}
	
The set $F$ of all elements of finite order of $H^+$ is a finite normal subgroup of $H$. Moreover there is a loxodromic element $h\in H^+$ such that the map $\sdirect{F}{\phi}{\group{h}} \to H^+$ that sends $(f,g)$ to $fg$ is an isomorphism, where $\phi\colon \group{h}\to \Aut(F)$ is conjugation on $F$.
	
\end{cor}

For a loxodromic element $g\in G$, we denote by $F(g)$ the set of all elements of finite order of $E^+(g)$. We say that $g$ is \emph{primitive} if its image in $E^+(g)/F(g)$ generates the quotient.

The \emph{$H$-invariant cylinder}, denoted by $C_H$, is the open $20\delta$-neighbourhood of all $10^3 \delta$-local  $(1, \delta)$-quasi-geodesics with endpoints $\xi$ and $\eta$ at infinity.

\begin{lem}[{Invariant cylinder; \cite[Lemma~3.13]{coulon_partial_2016}}]
	\label{L: invariant cylinder}
	Let $H\le G$ be a loxodromic subgroup. Then $C_H$ is invariant under the action of $H$ and strongly quasiconvex. In particular, $C_H$ is $2\delta$-quasi-convex. 
\end{lem}

\subsection{Fellow-traveling constant} 

The \emph{axis} of $g\in G$ is the set 
$$
A_g = \{x \in X \mid |gx-x|\le \tlen{g}+8\delta\}.
$$
\begin{rem}\label{R:axis-cylinder} If $g$ is loxodromic, then $ C_{\langle g\rangle}\subset A_g^{+52\delta}$. If  $\|g\|^{\infty}>10^3 \delta$, then $A_g\subset  C_{\langle g\rangle}$, \cite[Lemma 2.33]{coulon_geometry_2014}.
\end{rem}

\begin{lem}[\hspace{-0.15mm}{\cite[Proposition ~2.3.3]{delzant_courbure_2008};\cite[Proposition~2.28]{coulon_geometry_2014}}]
	
	\label{lem:axis-isometry} 
	
	Let $g\in G$. Then $A_g$ is $10\delta$-quasi-convex and $\group{g}$-invariant. Moreover, for every $x\in X$,
	$$\tlen{g}+2d(x,A_g)-10\delta \le |gx-x| \le \tlen{g}+2d(x,A_g)+10\delta.$$
	
\end{lem}

Let $g\in G$ be a loxodromic element and recall that $E(g)$ is the maximal loxodromic subgroup containing $g$.  
We write $u\sim_g v$ if and only if $u^{-1}v \in E(g)$, for every $u,v\in G$. In other words,  $u\sim_g v$ if and only if $uE(g)=vE(g)$.
The \emph{fellow travelling constant} of $g$ is
$$\Delta(g)=\sup  \setc{\diam(uA_g^{+20\delta}\cap vA_g^{+20\delta})}{ u,v\in G, u\not\sim_g v}.$$

\begin{lem}[\hspace{-0.15mm}{\cite[Proof of Proposition~6.29]{dahmani_hyperbolically_2017}}]
	\label{lem:acylindricity-intersection-axis} 
	 If $g\in G$ is loxodromic, then
	$$\Delta(g)\le \kappa+(N+2)\|g\|^{\infty}+100\delta.$$
\end{lem}

We also note the following fact:
\begin{lem}\label{L: axis} Let $g\in G$ with $\|g\|^{\infty}>10^3 \delta$. Then, for all $m\geqslant 1$, we have :
$$ \Delta (g^m)\leqslant \Delta(g) + 150\delta.$$
\end{lem}
\begin{proof}
By Remark \ref{R:axis-cylinder}, $A_{g^m}\subset C_{\langle g^m\rangle}=C_{\langle g\rangle}\subset A_g^{+52\delta}$. Thus
$$\Delta(g^m)\leq \diam\left(A_{g}^{+72\delta} \cap A_{ugu^{-1}}^{+72\delta}\right).$$ 
Quasi-convexity of $A_g$ now yields the claim, see  \cite[Lemma~2.2.2~(2)]{delzant_courbure_2008}. 
\end{proof}

\subsection{$\ell^{\infty}$-Energy.} 

To measure the displacement of a finite subset of isometries $U $ of $X$ we define the $\ell^{\infty}$-energy of $U$ at $x$ and the $\ell^{\infty}$-energy of $U$ as 
$$
L(U,x) = \max_{u\in U} |ux-x|, 
\quad \text{and} \quad
L(U)=\inf_{x\in X} L(U,x).
$$
The point $x$ is \emph{almost-minimizing the $\ell^{\infty}$-energy of} $U$ if $L(U,x)\le L(U)+\delta$. 

The translation length and the $\ell^{\infty}$-energy are related as follows. For every $g\in U$,
\begin{equation}	
	\label{eqn:energy and translation length}
	\stlen g \le \tlen{g} \le L(U).
\end{equation}
Moreover, we have the following.

\begin{prop}[\hspace{-0.15mm}{\cite[Proposition 3.2]{koubi_croissance_1998}}]\label{P:koubi} Let $U$ be a finite set of isometries of a hyperbolic space. If $L(U)>100\delta$, then $U^2$ contains a loxodromic isometry.
\end{prop}

If $U\subset G$ is symmetric and not contained in an elementary subgroup, Proposition \ref{P:koubi} can be used to find an independent pair of loxodromic elements in ping pong configuration. This yields that symmetric sets of sufficiently large energy have exponential growth.

In the next remark we explain that uniform uniform growth of hyperbolic groups follows from the results of \cite{koubi_croissance_1998}, and note that the estimates on the uniform uniform exponential growth rate so obtained depend on the hyperbolicity constant of the Cayley graph. 

\begin{rem}\label{R:proof_koubi} Let $G$ be a $\delta$-hyperbolic group. We want to argue that $G$ has uniform uniform exponential growth.  We fix a finite generating set such that the Cayley graph with respect to this generating set is $\delta$-hyperbolic. We denote this Cayley graph of $G$ by $X$. 

Let $S$ be a finite symmetric subset of $G$ that does not generate an elementary subgroup. As $S$ is symmetric, $S^2$ contains the identity, hence, $S^{2n}\subset S^{2n+2}$, for all $n\geqslant 1$. In fact, the inclusion of $S^{2n}$ in $S^{2n+2}$ has to be strict, as otherwise $\langle S\rangle$ would be a finite group. Thus $|S^{2n}|$ grows at least linearly in $n$.

Let $m_0$ be the cardinality of the ball of radius $100\delta$ in $X$, and let $U:=S^{2m_0+2}$. Then there is $u\in U$ such that $|ux-x|>100\delta$, for any $x\in X$, as otherwise $Ux\subset B(x,100\delta)$. In other words, the $\ell^{\infty}$-energy $L(U)>100\delta$. \cite[Proposition 3.2]{koubi_croissance_1998} now yields a loxodromic element $u$ in $U^2$. As $S$ does not generate an elementary subgroup, there is $s\in S$ such that $sus^{-1}$ and $u$ are independent loxodromic elements. Note that both these elements are in $SU^2S=S^2U^2=S^{4m_0+6}$.

It remains to argue that there is a constant $m_1$ (that does only depend on $G$) such that $su^{m_1}s^{-1}$ and $u^{m_1}$ freely generate a free subgroup. Then, as these two elements are in $S^{m_1(4m_0+6)}$, 
$$\omega(S)\geqslant \limsup_{l\to \infty} \frac{1}{m_1(4m_0+6)l}\log(|S^{m_1(4m_0+6)l}|)\geqslant  \frac{1}{m_1(4m_0+6)} \log 3.$$ To see the claim one can follow the argument of \cite[Proposition 5.5]{koubi_croissance_1998}. Indeed, first raise $sus^{-1}$ and $u$ to the power $n_0$ of \cite[Lemma 5.3]{koubi_croissance_1998}, so that $u'=u^{n_0}$ and $v=su^{n_0}s^{-1}$ each stabilise a bi-infinite geodesic line denoted by $L_{u'}$ and $L_{v}$.  Note that $n_0$ depends on $\delta$ and the cardinality of the balls of radius $8\delta$ in $X$. As observed in the first paragraph of \cite[Proposition 5.5]{koubi_croissance_1998}, these loxodromic isometries have translation length at least $200\delta$.  If $d(L_{u'},L_v)\leqslant 8\delta$,  the argument of \cite[Proposition 5.5]{koubi_croissance_1998} implies that $u'^{n_2}$ and $v^{n_2}$ freely generate a free subgroup, where $n_2$ is the constant defined in the proof of \cite[Proposition 5.5]{koubi_croissance_1998} .  Note that the constant $n_2$ only depends on the cardinality of balls of radius $24\delta$ in $X$. Otherwise, let $p$ be the midpoint on a shortest geodesic between $L_{u'}$ and $L_v$. In this case, one verifies from the definition of hyperbolicity that $(u'^{\pm 1}p,v^{\pm 1}p)_p\leqslant 2\delta$ and that $\min\{|u'^{\pm 1}p-p|,|v^{\pm 1}p-p|\}\geqslant 190\delta$. Then, as in the last paragraph of the proof of \cite[Proposition 5.5]{koubi_croissance_1998}, \cite[Lemme 2.4]{koubi_croissance_1998} implies that $u'$ and $v$ freely generate a free subgroup. 
We fix $m_1=n_2n_0$. It depends on $\delta$ and the cardinality of the respective balls in $X$. 
\end{rem} 

In our general setting the argument of the previous remark does not apply, but we still have the following. 

\begin{lem}[\hspace{-0.15mm}{\cite[Theorem 13.1]{breuillard_joint_2021}}]
	\label{lem:breuillard-fujiwara}
	Let $G$ act by isometries on a hyperbolic space. Let $U\subset G$ be a finite symmetric subset containing the identity. Then one of the following hold:
\begin{enumerate}
	\item $L(U)\le 10^4 \delta$. 
	\item The subgroup $\group{U}$ is virtually cyclic and contains a loxodromic element.
	\item $\omega(U) \ge \frac{1}{10^5} \log 2$.
\end{enumerate}
\end{lem}

\section{Growth of maximal loxodromic subgroups.}

Recall that $\kappa\geqslant \delta$ and $G$ acts $(\kappa,N)$-acylindrically on the $\delta$-hyperbolic space $X$.

\begin{prop}
	
	\label{prop:growth-loxodromic-subgroups}
	
Let $U \subset G$ be a finite symmetric subset. Let $g\in G$ be a primitive loxodromic element. Then, for every $n \ge 1$,
	$$|U^n\cap E(g)| \le 2N\qty(\frac{L(U)}{\stlen{g}}4n+1).$$
	
\end{prop}

 This is proved as in the case of hyperbolic groups, and we follow the argument of \cite[p.~484]{antolin_degree_2017}. 
First, we treat the cyclic group generated by a loxodromic isometry.

\begin{lem}
	
	\label{lem:growth-cyclic-loxodromic-subgroups}
	
Let $U\subset G$ be a finite symmetric subset. Let $g\in G$ be a loxodromic element. Then, for every $n \ge 1$,
	$$|U^n\cap\group{g}|\le \frac{L(U)}{\stlen{g}}2n+1.$$
	
\end{lem}

\begin{proof}
	
	Let $n \ge 1$. 
	Since $U$ is symmetric,
	$$|U^n\cap \group{g}|=|\setc{k \in \Z}{g^k \in U^n}|\le 2|\setc{k\in \N\setminus \{0\}}{g^k\in U^n}|+1.$$
	Let $k \ge 1$ such that $g^k \in U^n$. 
Then $\|g^k\|^{\infty}\le L(U^n)$, see \eqref{eqn:energy and translation length}. By the triangle inequality, $L(U^n)\le nL(U)$. Hence,
	$$k=\frac{\|g^k\|^{\infty}}{\|g\|^{\infty}} \le \frac{L(U)}{\stlen{g}}n.$$ 
This yields the claim. 	
\end{proof}

\begin{proof}[Proof of Proposition \ref{prop:growth-loxodromic-subgroups}] Let $g$ be primitive. 	
	 Recall that $F(g)$ is set of all elements of finite order of $E^+(g)$ and that $\sdirect{F(g)}{\phi}{\group{g}}$ is isomorphic to $ E^+(g)$.  
	  Let $n \ge 1$. Let $E_0$ be a set of representatives of $E(g)/\group{g}$. Then 
	$$|U^n\cap E(g)|=\sum_{r\in E_0} |U^n\cap r\group{g}|.$$
	
	First we estimate $|E_0|$. By definition, $[E(g)\colon E^+(g)]\le 2$. As $E^+(g)=\sdirect{F(g)}{\phi}{\group{g}}$, we have $[E^+(g)\colon \group{g}]=|F(g)|$. By acylindricity, $|F(g)|\le N$.   Thus,  
	$|E_0|\le 2N.$
	
	Now we estimate $|U^n\cap r\group{g}|$ for $r\in E_0$. We may assume that $U^n\cap r\group{g}$ is non-empty. Then there exist $s\in U^n\cap r\group{g}$. In particular $r\group{g}=s\group{g}$. Hence,
	$$|U^n\cap r\group{g}|=|U^n\cap s\group{g}|=|s(s^{-1}U^n\cap \group{g})|=|s^{-1}U^n\cap \group{g}|.$$
	Since $U$ is symmetric, $s^{-1}\in U^n$. Thus,  $s^{-1}U^n\subset U^{2n}$. Therefore,
	$$|U^n\cap r\group{g}|\le |U^{2n}\cap \group{g}|.$$
	By Lemma \ref{lem:growth-cyclic-loxodromic-subgroups}, the latter is bounded by 
	$ \frac{L(U)}{\stlen{g}}4n+1.$
	Consequently,
	$$|U^n\cap r\group{g}|\le \frac{L(U)}{\stlen{g}}4n+1.$$
		Combining the above estimates yields the claim. 
\end{proof}

Given $U\subset G$ and a loxodromic element $g\in G$, we fix a set of representatives $U(g)$ of the equivalence relation induced on $U$ by $\sim_g$, that is, $u\sim_g v$ if and only if $u^{-1}v\in E(g)$, for every $u,v\in U$.  
  We obtain the following.

\begin{cor}
	
	\label{cor:growth-loxodromic-subgroups}
	Let $G$ be a group acting $(\kappa,N)$-acylindrically on a hyperbolic space $X$. Let $U\subset G$ be a finite symmetric subset. Let $g\in G$ be a primitive loxodromic element. Let
	$$a_0=2N\qty(\frac{L(U)}{\stlen{g}}8+1).$$
	Then,
	$$|U(g)|\ge \frac{1}{a_0}|U|.$$
	
\end{cor}

\begin{proof}
	
Let $u,v\in U$ such that $u \sim_g v$. By definition, $u^{-1}v\in E(g)$. As $U$ is symmetric, $u^{-1}v\in U^2$. Therefore, $v\in u(U^2\cap E(g))$. Note that $|u(U^2\cap E(g))|=|U^2\cap E(g)|$. Consequently, each $u\in U(g)$ has at most $|U^2\cap E(g)|$ elements in its equivalence class. By Proposition \ref{prop:growth-loxodromic-subgroups}, $|U^2\cap E(g)|\le a_0$, hence, the claim follows.	
\end{proof}

\section{Reduced subsets}
\label{sec:reduced-subsets}

Recall that $G$ acts by isometries on the $\delta$-hyperbolic space $X$. If $U\subset G$ we denote by $U^{-1}$ the set of the inverses of the elements of $U$. 

\begin{df}	
	\label{df:reduced-subset}
	
	Let $\alpha \geq 3\delta$. A finite subset $U\subset G$ is \emph{$\alpha$-reduced at $p\in X$} if $U\cap U^{-1}=\varnothing$ and for every pair of distinct $u_1, u_2 \in U\sqcup U^{-1}$,
	$$
	(u_1p, u_2p)_p<\frac{1}{2}\min\{|u_1p-p|,|u_2p-p|\}-\alpha-50\delta.
	$$
	
\end{df}

\begin{rem}
\label{rem:reduced-subset}
If $U\subset G$ is $\alpha$-reduced at $p\in X$, then $|up-p|>2(\alpha+50\delta)$, for every $u\in U\sqcup U^{-1}$.
\end{rem}

We clarify some vocabulary. Let $U\subset G$ be a subset. We write $w_1\equiv w_2$ to express letter-for-letter equality of words $w_1$ and $w_2$ over $U\sqcup U^{-1}$. We denote by $\F(U)$ the free group on $U$. Every element $w\in \F(U)$ is represented as a reduced word $w\equiv u_1\cdots u_n$ with $u_i\in U\sqcup U^{-1}$. The number $n$ in such a reduced representation of $w\in \F(U)$ is called the \emph{length} of $w$, denoted by $|w|_U$. The \emph{natural homomorphism} $\psi\colon \F(U)\to G$ is the evaluation of the elements of $\F(U)$ in $G$. We usually abuse notation and identify $w$ with its image under $\psi$. 
Also, note that $|.|_U$ coincides with the word length on $\F(U)$. 

\subsection{Broken geodesics}

The next lemma is used to produce quasi-geodesics.  For any integers $n\leqslant m$, we let $\zinterval{n}{m}=\{t\in \mathbb{Z}\mid n\leqslant t\leqslant m\} $.  

\begin{lem}[{Broken Geodesic Lemma \cite[Lemma~1]{arzhantseva_lower_2006}}]
	\label{lem:broken geodesic}
	Let $n\ge 2$. Let $x_0,\cdots,x_n$ be a sequence of $n+1$ points of $X$. Assume that
	\begin{equation}
		\label{eqn:broken-geodesic}
	(x_{i-1},x_{i+1})_{x_i}+(x_i,x_{i+2})_{x_{i+1}}<|x_i-x_{i+1}|-3\delta,
	\end{equation}
	for every $i\in \zinterval{1}{n-2}$. Then the following holds.
	\begin{enumerate}[label=(\roman*)]
		\item $\displaystyle |x_0-x_n|\ge \sum_{i=0}^{n-1}|x_i-x_{i+1}|-2\sum_{i
			=1}^{n-1}(x_{i-1},x_{i+1})_{x_i}-2(n-2)\delta.$
		\item $(x_0,x_n)_{x_j}\le (x_{j-1},x_{j+1})_{x_j}+2\delta$, for every $j\in\zinterval{1}{n-1}.$
		\item The geodesic $[x_0,x_n]$ lies in the $5\delta$-neighbourhood of the broken geodesic $\gamma=[x_0,x_1]\cup\cdots\cup[x_{n-1},x_n]$, while $\gamma$ is contained in the $r$-neighbourhood of $[x_0,x_n]$, where
				$$r=\sup_{1\le i\le n-1} (x_{i-1},x_{i+1})_{x_i}+14\delta.$$
	\end{enumerate}
\end{lem}

We apply Lemma \ref{lem:broken geodesic} as follows. Let $\alpha\geqslant 3\delta$ and let $U\subset G$ be an $\alpha$-reduced subset at $p\in X$. Let $n\ge 2$. Given a reduced word $w\equiv u_1\cdots u_n \in \F(U)$, we fix    
	$$x_0=p,\quad x_1=u_1p,\quad x_2=u_1u_2p,\quad \cdots,\quad x_n=u_1\cdots u_np.$$
 Then we have the following. 	

\begin{prop}
	\label{prop:broken geodesic reduced subset}
Let $n\ge 2$ and let $w\equiv u_1\cdots u_n \in \F(U)$ be a reduced word. Then 
	\begin{enumerate}[label=(\roman*)]
		\item $(x_{i-1},x_{i+1})_{x_i}+(x_{i},x_{i+2})_{x_{i+1}}<|x_i-x_{i+1}|-2(\alpha+50\delta)$,
		for every $i\in\zinterval{1}{n-2}$.
		\item 	$|wp-p|\ge \frac{1}{2}|u_1p-p|+\frac{1}{2}|u_np-p|+2(n-1)(\alpha+40\delta)+4(n-1)\delta$,
		\item $(p,wp)_{x_i}\le (u_i^{-1}p,u_{i+1}p)_p +2\delta$, for all $i\in \zinterval{1}{n-2}$.
	\end{enumerate}
\end{prop}

\begin{proof}
	 Let $i\in\zinterval{1}{n-2}$. We have
	$(x_{i-1},x_{i+1})_{x_i}=(u_{i}^{-1}p,u_{i+1}p)_p$ 
	and $|x_i-x_{i+1}|=|p-u_{i+1}p|$. As $w$ is reduced, $u_i^{-1}\neq u_{i+1}$ and $u_{i+1}^{-1}\neq u_{i+2}$. As $U$ is $\alpha$-reduced at $p$, 
	$$(u_i^{-1}p, u_{i+1}p)_p<\frac{1}{2}|u_{i+1}p-p|-\alpha-50\delta,\quad
	(u_{i+1}^{-1}p, u_{i+2}p)_p<\frac{1}{2}|u_{i+1}^{-1}p-p|-\alpha-50\delta.$$
	Adding the two above inequalities we obtain assertion (i). 
	
Applying (i) and Lemma \ref{lem:broken geodesic} (i) to the sequence $x_0,\cdots, x_n$, we obtain
	\begin{align*}
		|wp-p|&\ge |u_1p-p|+\sum_{i=2}^{n-1}|u_{i}p-p|+|u_np-p|\\
		&-(u_1^{-1}p,u_{2}p)_p
		-\sum_{i=2}^{n-1}[(u_i^{-1}p,u_{i+1}p)_p+(u_{i-1}^{-1}p,u_{i}p)_p]-(u_{n-1}^{-1}p,u_np)\\
		&-2(n-2)\delta.
	\end{align*}
	Since $U$ is $\alpha$-reduced at $p$,
		$$(u_1^{-1}p,u_{2}p)_p<\frac{1}{2}|u_1p-p|-\alpha-50\delta,\quad (u_{n-1}^{-1}p,u_np)<\frac{1}{2}|u_np-p|-\alpha-50\delta,$$
and by assertion (i) we have that
$$\sum_{i=2}^{n-1}[(u_i^{-1}p,u_{i+1}p)_p+(u_{i-1}^{-1}p,u_{i}p)_p]< \sum_{i=2}^{n-1}|u_ip-p|-2(n-2)(\alpha+50\delta).$$
This yields assertion (ii). Assertion (iii) follows directly from Lemma \ref{lem:broken geodesic}(ii).
\end{proof}

\subsection{Quasi-isometric embedding of a free group}

Recall that $L(U,p)$ denotes the $\ell^{\infty}$-energy of $U\subset G$ at $p\in X$.

\begin{prop}
	\label{prop:reduced subset qi embedding} Let $\alpha\geqslant 3\delta$. 
	Let $U\subset G$ be an $\alpha$-reduced subset at $p\in X$. Then, for every $w\in \F(U)$, we have
	$$2\alpha|w|_U\le |wp-p|  \le L(U,p)|w|_U.$$
	In particular, the natural homomorphism $\psi\colon \F(U)\to G$ is injective.
\end{prop}

\begin{proof}
	Let $w\equiv u_1\cdots u_n \in \F(U)$ be a reduced word. If $n=0$ or if $n=1$, the result is by Definition. Let $n\ge 2$. By the triangle inequality, $|wp-p|\le L(U,p)n$. By Proposition \ref{prop:broken geodesic reduced subset} (ii) 
	$$|wp-p|\ge \frac{1}{2}|u_1p-p|+\frac{1}{2}|u_np-p|+2(n-1)(\alpha+\delta)+2\delta.$$
	Combined with Remark \ref{rem:reduced-subset} this yields that $|wp-p|\ge 2\alpha n.$
		In particular, if $w\in \F(U)$ is non-trivial, then $|wp-p|>0$ and $w\neq 1$ in $G$. In other words, $\psi $ is injective.
\end{proof}

\subsection{Geodesic extension property}

We prepare for running the counting argument of \cite[Section 3]{coulon_product_2022} for strongly reduced sets in our setting of $\alpha$-reduced sets. For this purpose we need the following version of \cite[Lemma~3.2]{coulon_product_2022}. 

\begin{prop}
	\label{prop:geodesic extension property} Let $\alpha\geqslant 3\delta$. 
Let $U\subset G$ be an $\alpha$-reduced subset at $p$. Let $w\equiv u_1\cdots u_m$ and $w'\equiv u_1'\cdots u_{m'}'$ be two elements of $\F(U)$ in reduced form. Then $U$ satisfies the geodesic extension property, that is, if $$(p,w'p)_{wp}\leqslant \frac{1}{2}|u_mp-p|,$$
	then $w$ is a prefix of $w'$.
\end{prop}

\begin{rem}
The \emph{geodesic extension property} has the following meaning: if the geodesic $[p,w'p]$ extends $[p,wp]$ as a path in $X$, then $w'$ extends $w$ as a word over $U\sqcup U^{-1}$.
\end{rem}

Up to some minor technical differences, the proof of Proposition \ref{prop:geodesic extension property} is the same as the proof of \cite[Lemma~3.2]{coulon_product_2022}. We include it here. 

\begin{proof}
	The proof is by contradiction.  Assume that $w$ is not a prefix of $w'$. Let $r<m$ be the largest integer such that $u_i=u_i'$, for every $i\in\zinterval{1}{r-1}$. We let 
	$$q=u_1\cdots u_{r-1}p=u_1'\cdots u_{r-1}'p.$$ 
	We show that $\min\{(p,q)_{wp},(q,w'p)_{wp}\}> \frac{1}{2}|u_mp-p|+\delta.$ Hyperbolicity then implies the proposition. 
	The definition of Gromov product implies that 
	\begin{equation}
		\label{eqn:geodesic-extension2}	
		(p,q)_{wp}=|wp-q|-(p,wp)_q,\quad (q,w'p)_{wp}=|wp-q|-(wp,w'p)_q.
	\end{equation}
	We estimate $|wp-q|$, $(p,wp)_q$, and $(wp,w'p)_q$. Firstly, by Proposition \ref{prop:broken geodesic reduced subset} (ii),   
	\begin{align}
		|wp-q|\ge \frac{1}{2}|u_rp-p|+\frac{1}{2}|u_mp-p|+2(m-r)(\alpha+\delta).\label{cla:geodesic-extension}
	\end{align}
Next by Proposition \ref{prop:broken geodesic reduced subset}(iii), and as $U$ is $\alpha$-reduced,
	\begin{align}
		(p,wp)_q<\frac{1}{2}|u_rp-p|. \label{eqn:geodesic-extension4} 
	\end{align}
Finally, we prove that 
	\begin{align}
		(wp,w'p)_q<\frac{1}{2}|u_rp-p|. \label{eqn:geodesic-extension5} 
	\end{align}
		Indeed, if $r-1=m'$, then $w'p = q$ and \eqref{eqn:geodesic-extension5} holds. Let $r-1<m'$. By choice of $r$ we have that $u_r\neq u_r'$. 
		For simplicity, denote
	$$t=u_1\cdots u_{r}p\quad\text{and}\quad t'=u_1'\cdots u_{r}'p.$$ 
	By hyperbolicity 
		$$\min\{(t,wp)_q,(wp,w'p)_q,(w'p,t')_q\}\le (t,t')_q+2\delta=(u_rp,u_r'p)_p+2\delta.$$
		Since $U$ is $\alpha$-reduced at $p$,
				\begin{equation}
			\label{eqn:geodesic-extension3}
			\min\{(t,wp)_q,(wp,w'p)_q,(w'p,t')_q\}< \frac{1}{2}\min\{|u_rp-p|,|u_r'p-p|\}-\alpha.
		\end{equation}
		We prove that the minimum of \eqref{eqn:geodesic-extension3} is attained by $(wp,w'p)_q$. Indeed,  
		$$(t,wp)_q=|q-t|-(q,wp)_t.$$
		By definition, $|q-t|=|u_rp-p|.$ Recall that $m\ge r+1$. By Proposition \ref{prop:broken geodesic reduced subset}(iii), 
		$$(q,wp)_t\le (u_r^{-1}p,u_{r+1}p)_p+2\delta<\frac{1}{2}|u_rp-p|-\alpha.$$
				Consequently,
		 $(t,wp)_q > \frac{1}{2}|u_rp-p|+\alpha.$ 
		Thus, the minimum of \eqref{eqn:geodesic-extension3} cannot be attained by $(t,wp)_q$. Similarly, it cannot be attained by $(w'p,t')_q$. This proves Equation \eqref{eqn:geodesic-extension5}. 
	Finally, combining Equations (\ref{eqn:geodesic-extension2}-\ref{eqn:geodesic-extension5}), yields the claim. 
	\end{proof}
	
	\subsection{Producing reduced subsets}

Recall that the action of $G$ on $X$ is $(\kappa,N)$-acylindrical and that $\Delta(g)$ is the fellow travelling constant of a loxodromic $g\in G$ (Subsection \ref{subsec:group-action}). Also, if $U$ is a finite symmetric subset of $G$, then $U(g)=U/\sim_g$ where $u\sim_gv$ if and only if $uv^{-1}\in E(g)$.  

\begin{prop}
		\label{prop:producing-reduced-subsets}	
	Let $\alpha\geq 3 \delta$. Let $U\subset G$ be a finite symmetric subset containing a loxodromic isometry $g$ such that $\stlen{g}>10^3\delta$.  Let $p\in X$. Let
	$$b_0 = \frac{2}{\stlen{g}} \qty [\Delta(g)+5L(U,p)+104\delta+\alpha ].$$
	Then for every $b\ge b_0$, the set $S=\setc{ug^bu^{-1}}{u\in U(g)}$ satisfies the following: 
	\begin{enumerate}
		\item $S\subset U^{b+2}$. 
		\item $|S|=|U(g)|$.
		\item $S$ is $\alpha$-reduced at $p$.
	\end{enumerate}
	
\end{prop}
The proof is by standard arguments. See for example \cite{koubi_croissance_1998}, \cite{arzhantseva_lower_2006} or \cite{fujiwara_rates_2021}.  We give it here to justify the given formula for $b_0$.

\begin{proof}
	The conclusions (1) and (2) are immediate. We now prove (3). By definition, $S\cap S^{-1} = \varnothing$. Let $i\in\zinterval{1}{2}$. Let $u_i\in U(g)$. Let $\varepsilon_i\in \{-1,1\}$. Assume that $u_1g^{\varepsilon_1b}u_1^{-1}$ and $u_2g^{\varepsilon_2b}u_2^{-1}$ are distinct.
	
	\underline{\emph{Case} $u_1=u_2$}. In this case, $\varepsilon_1=-\varepsilon_2$. Let $h=u_1g^{\varepsilon_1b}u_1^{-1}$. It is enough to prove that
	$$(hp,h^{-1}p)_p\le \frac{b}{2}\stlen{g}-\alpha-50\delta.$$
	This follows from Lemma \ref{lem:quasi-line} applied to an $h$-invariant quasi-geodesic line.

	\underline{\emph{Case} $u_1\neq u_2$}. Then $u_1\not\sim_g u_2$, which means that $u_1^{-1}u_2\not \in E(g)$. By 
 Lemma \ref{lem:axis-isometry},
\begin{align}
d(p,A_g)\le \frac{1}{2}|gp-p|+5\delta\le \frac{1}{2}L(U,p)+5\delta. \label{E:producing-reduced-1}
\end{align}		
		Now let $x_i = u_ip$ and $y_i = u_ig^{\varepsilon_ib}p$. We then have that 
\begin{align}
\diam([x_1,y_1]^{+8\delta}\cap[x_2,y_2]^{+8\delta})\le \Delta(g)+L(U,p)+54\delta. \label{E:producing-reduced-2}
\end{align}
Indeed, let $\sigma=d(p,A_g)+10 \delta$, so that  
		$\max \qty {d(x_i,u_iA_g), d(y_i, u_iA_g)}\le \sigma.$
		As $A_g$ is $10\delta$-quasi-convex (Lemma \ref{lem:axis-isometry}), the subset $u_iA_g^{+\sigma}$ is $2\delta$-quasi-convex. Consequently, 	
		$[x_i,y_i]\subset u_iA_g^{+\sigma+2\delta}.$
		Therefore,  and by quasi-convexity, {\cite[Lemma~2.2.2~(2)]{delzant_courbure_2008}},
		\begin{align*}
		\diam([x_1,y_1]^{+8\delta}\cap[x_2,y_2]^{+8\delta}) & \le \diam(u_1A_g^{+\sigma+10\delta}\cap u_2A_g^{+\sigma+10\delta})\\
		& \le \diam(u_1A_g^{+13\delta}\cap u_2A_g^{+13\delta})+2(\sigma+10\delta)+4\delta \\
		& \le \Delta(g)+2(\sigma+10\delta)+4\delta.
		\end{align*}
 Combining the above estimations with \eqref{E:producing-reduced-1}, we obtain \eqref{E:producing-reduced-2}.
	
	Finally, let $s_i=u_ig^{\varepsilon_ib}u_i^{-1}$. We estimate $(s_1p,s_2p)_p$:
		by the triangle inequality, 
		$$(s_1p,s_2p)_p \le \frac{1}{2} \qty (|x_1-y_1|+|x_2-y_2|-|y_1-y_2|)+\frac{3}{2} \qty (|x_1-p|+|x_2-p|).$$
		Combining \eqref{E:producing-reduced-2} with Lemma \ref{lem:intersection-four-points}, we obtain
		$$|x_1-y_1|+|x_2-y_2|-|y_1-y_2| \le |x_1-x_2|+2(\Delta(g)+L(U,p)+54\delta).$$
		By the triangle inequality,
		$|x_1-x_2| \le 2 L(U,p).$
		 We obtain that 
		\begin{align} \label{E:producing-reduced-3}
		(s_1p,s_2p)_p\le \Delta(g)+5L(U,p)+54\delta.
		\end{align}
	Finally, note that
$
\min \qty { |s_1p-p|, |s_2p-p| }  \ge {b}\stlen{g}.
$
	By the choice of $b_0$, 
	$$(s_1p,s_2p)_p<\frac{1}{2} \min \qty { |s_1p-p|, |s_2p-p| }-\alpha-50\delta,$$
	which shows that $S$ is $\alpha$-reduced. 	
\end{proof}

\section{Ping-pong in powers of $U$}

We recall that $N>0$, $\delta >0$ and $\kappa\geq \delta$ and that $G$ acts $(\kappa,N)$-acylindrically on a $\delta$-hyperbolic space.  The goal of this section is to find a generating set of a free subgroup in some power of $U\subset G$ whose cardinality and energy is comparable to that of $U$.

\subsection{The case of large energy}

We first prove a variant of Theorem \ref{thm:growth-trichotomy-intro}. It immediately implies Theorem \ref{thm:growth-trichotomy-intro}.

\begin{prop}\label{P:Pingpong_intro}  Let $C=10^6(N+1)$.  There is a number  $n>C^3$ such that the following holds. Let $U\subset G$ be a finite symmetric subset  such that $2\cdot 10^3\kappa < L(U)$. If $U$ is not contained in an elementary subgroup,  then there are $S\subset U^{10^7n}$ and $p\in X$ such that 
\begin{enumerate}
\item $S$ is $\alpha$-reduced at $p$, for $\alpha=200\delta$,
\item $|S| \geqslant \frac{1}{C^2}|U^n|$,
\item $L(S,p)\le L(U^{10^7n},p)\leqslant 10^{12} C^3 L(U)$.
\end{enumerate}
In particular $S$ freely generates a free subgroup and $|S|>C$. 
\end{prop}

\begin{rem} The exponent $n$ of Proposition \ref{P:Pingpong_intro} does depend on $N$, but it does not depend on $U$. Neither does it depend on $\kappa$ or $\delta$. 
\end{rem}

We now prove this proposition. Let $U\subset G$ be a finite symmetric subset that is not contained in an elementary subgroup. Moreover, we suppose that $2\cdot 10^3\kappa< L(U)$. 

As $L(U) >100\delta$, then $U^2$ contains a loxodromic element \cite[Proposition 3.2]{koubi_croissance_1998}. We may choose this loxodromic element of length $\|h\|^{\infty}>L(U)/2$, \cite[Lemma 7]{arzhantseva_lower_2006} \cite[Theorem 5.7]{breuillard_joint_2021} \cite[Lemma 2.7]{fujiwara_rates_2021}. This yields the following lemma.

Let us set $L=L(U)$. 
\begin{lem}\label{L:power-large-energy-2} There is $h\in U^{2}$  such that
 $$L/2 < \|h\|^{\infty}.$$
 In particular, $h$ is loxodromic and  $\|h\|^{\infty}\leqslant 2L$.
\end{lem}

This implies the following. We let $C=10^6(N+1)$.
\begin{prop}\label{P:long-hyperbolic-short-intersection} Let $n_1=4 C^2$. There is $h\in U^{n_1}$ such that 
\begin{itemize}
\item  $L (U^{n_1}) \leqslant 4C^2L$
\item $h$ is loxodromic and $\|h\|^{\infty}\geqslant C^2L$
\item for all $m\geqslant 1$,  $\Delta(h^m)\leqslant C L$.
\end{itemize}
\end{prop}
\begin{proof} Let $n_1=4 C^2 $. The first assertion is by the triangle inequality. Let $h=g^{2C^2} \in U^{n_1}$.  Let  $g\in U^{2}$ be the loxodromic element of length $\|g\|^{\infty}> L/2$ and $\|g\|^{\infty}\leqslant 2L$ given by the previous lemma. The second assertion follows from the fact that $\|g^m\|^{\infty}=m\|g\|^{\infty}$. Finally let $m>0$. Note that $||g||^{\infty}>10^3\delta$. By Lemmas \ref{L: axis} and \ref{lem:acylindricity-intersection-axis},   
$\Delta(h^m) =\Delta(g^{2C^2m}) \leq \Delta(g)+150\delta \leq (N+2) \|g\|^{\infty} +\kappa + 250\delta \leqslant C L.$ This proves the third assertion.
\end{proof}

Let $n_1$ and $h\in U^{n_1}$ be given by Proposition \ref{P:long-hyperbolic-short-intersection}. As $U$ is not contained in an elementary subgroup, there is $u\in U$ such that $h$ and $uhu^{-1}$ do not generate an elementary subgroup. Moreover, $g=h^{100}uh^{1000}u^{-1}h^{2000}uh^{100}u^{-1}h^{100}\in U^{3304n_1}$ is a primitive loxodromic element and $\|g\|^{\infty}>C^2L$. See \cite[Lemma 2.8]{fujiwara_rates_2021}. Moreover, $\|g\|^{\infty}\leqslant L(U^{3304n_1})\leqslant 13216 C^2L$. Thus, by Lemmas \ref{L: axis} and \ref{lem:acylindricity-intersection-axis}, $\Delta(g^m)\leqslant C^3 L$, for all $m\geqslant 1$.
  This yields:

\begin{prop}\label{P:long-primitive-short-intersection} Let $n_2=3304n_1$. There is $h\in U^{n_2}$ such that 
\begin{itemize}
\item  $L (U^{n_2}) \leqslant 13216 C^2L$
\item $h$ is primitive, loxodromic and $\|h\|^{\infty}\geqslant C^2L$
\item for all $m\geqslant 1$,  $\Delta(h^m)\leqslant C^3 L$.
\end{itemize} 
\end{prop}

\begin{proof}[Proof of Proposition \ref{P:Pingpong_intro}]
Let $n_2>C^2$ and $h\in U^{n_2}$ be a loxodromic and primitive element given by Proposition \ref{P:long-primitive-short-intersection}.
 Let $n=Cn_2$. As $U$ is symmetric, $U^{n_2}$ contains the identity, hence, $h\in U^n$. Let 
 $S:=\{uh^{10^6C}u^{-1}\mid u\in U^{n}(h)\}\subset U^{10^7n}$ and let $p\in X$ be a point almost minimising the $\ell^{\infty}$-energy of $U^n$.  
Let $\alpha=200\delta$. By Proposition \ref{prop:producing-reduced-subsets}, $S$ is $\alpha$-reduced at $p$ and $|S|=|U^n(h)|$. Indeed, in our case, by choice of $C$, $h$ and $n$, 
\begin{align*}
b_0 & =\frac{2}{\|h\|^{\infty}}(\Delta(h)+5L(U^n,p)+104\delta + \alpha)\\
& \leqslant \frac{2}{C^2L}(C^3L+5\cdot 13216 C^3L+\delta + 104\delta+200\delta)\leqslant  10^6 C
\end{align*} 
By Corollary \ref{cor:growth-loxodromic-subgroups}, and as $\|h\|^{\infty}>L(U^{n})/13216 C$ and $C>10^6N$, we have that 
$$|S|=|U^{n}(h)|>\frac{|U^{n}|}{C^2}.$$
 As $U$ is not contained in a finite subgroup the growth of $|U^n|$ is at least linear in $n$. Thus $|U^{n}|>C^3$, which implies that $|S|>C$. 
\end{proof}

\subsection{The general case}

 Let $U\subset G$ be a finite symmetric subset that is not contained in an elementary subgroup. Moreover, we let $L > 10^4\kappa$  and suppose that $L \ge L (U)$. 
 As $U$ is not contained in an elliptic subgroup, the subgroup generated by $U$ contains a loxodromic element. More precisely, we have the following. 

\begin{lem}\label{L:power-large-energy} Let $a<L$. Then there is $n>0$, depending only on $U$ and $a$, such that 
$$a< L(U^{n}) \leqslant 2L.$$
\end{lem}
\begin{proof} If $L(U)>a$, $n=1$. Otherwise we proceed as follows. As $U$ is not contained in an elliptic subgroup, $U^m$ contains a loxodromic element, for some $m>0$. Then $L (U^{km})$ is at least linearly growing in $k>0$. Thus we can pick $n>1$ to be the maximal number such that  $L(U^{n-1})\leqslant a$. Then $L(U^{n})>a$, and $L(U^{n}) \leqslant 2L$ by the triangle inequality. 
\end{proof}

Let us now fix $a=L/2$, and let $n$ be the number given by Lemma \ref{L:power-large-energy}, so that 
$$L/2 \leqslant L(U^n)\leqslant 2L.$$ 
We can now apply Proposition \ref{P:Pingpong_intro} to $U^n$. This yields the following. 

\begin{prop}\label{P:Pingpong}  Let $C=10^6(N+1)$. For all $L >10^4 \kappa$ and for all finite symmetric subsets $U\subset G$ such that $L(U)\leqslant L$ and such that $U$ is not contained in an elementary subgroup,  there are $n>C^3$, $S\subset U^{10^7n}$ and $p\in X$ such that 
\begin{enumerate}
\item $S$ is $\alpha$-reduced at $p$, for $\alpha=200\delta$,
\item $|S| \geqslant \frac{1}{C^2}|U^n|$,
\item $L(S,p)\leqslant 10^{13} C^3 L$.
\end{enumerate}
In particular $S$ freely generates a free subgroup and $|S|>C$. 
\end{prop}
\begin{rem} Note that in contrast to Proposition \ref{P:Pingpong_intro}, the exponent $n$ now depends on $U$. As in Proposition \ref{P:Pingpong_intro} it also depends on $N$, but it does not depend on $\kappa$ or $\delta$. 
\end{rem}

\section{Small cancellation theory}
\label{subsec:small-cancellation-theory}

Let $G$ be a group acting by isometries on $X$. We recall that $X$ is a $\delta$-hyperbolic space.

\subsection{Loxodromic moving family.}
Let $\scrQ$ be a set of pairs $(H,Y)$, where $H$ is a loxodromic subgroup of $G$ and $Y=C_H$ its invariant cylinder. We further assume that $\scrQ$ is invariant under the action of $G$ defined by $g\cdot (H,Y)=(gHg^{-1},gY)$. 
In this case we refer to $\scrQ$ as of a  
 \emph{loxodromic moving family}. 
 The \emph{fellow travelling constant of} $\scrQ$ is
$$\Delta(\scrQ,X)=\sup\setc{\diam(Y_1^{+20\delta}\cap Y_2^{+20\delta})}{ (H_1,Y_1)\neq (H_2,Y_2)\in\scrQ}.$$
The \emph{injectivity radius of} $\scrQ$ is
$$T(\scrQ,X)=\inf\setc{\tlen{h}}{h\in H-\{1\}, (H,Y)\in\scrQ}.$$

 We denote $K=\normal{H\mid (H,Y)\in \scrQ}$ and $\q G= G/K$. We denote by $\pi\colon G \onto \q G$ the natural projection and write $\q g$ for $\pi(g)$ for short, for every $g\in G$. 
\begin{df}[Small cancellation condition]\label{D:small-cancellation}
	Let $\lambda>0$ and $\mu>0$. We say that $\scrQ$ satisfies the \emph{$C''(\lambda,\mu)$-small cancellation condition} if 
	\begin{enumerate}[label=(SC\arabic*)]	
		\item $\Delta(\scrQ,X)< \lambda T(\scrQ,X)$,
		\item $T(\scrQ,X)> \mu\delta$.
	\end{enumerate}
In this case we say that $\q G$ is a \emph{$C''(\lambda,\mu)$-small cancellation quotient}.
\end{df}

\begin{rem} Let $\scrQ$ be a loxodromic moving family under the $C''(\lambda,\mu)$-small cancellation condition, where $\lambda <1$ and $\mu>10^3$. Let $(H_1,Y_1)$ and $(H_2,Y_2)\in \scrQ$. 
Condition (SC1) implies that $Y_1=Y_2$ if and only if $H_1=H_2$. Condition (SC2) implies that all elements in the loxodromic subgroups  $H$, for $(H,Y) \in \scrQ$, are loxodromic. Thus, the groups $H$ are all cyclic.
Moreover, as $\scrQ$ is invariant under the action of $G$ by conjugation, the groups $H$ are normal in the the stabiliser of $Y$, that is, the maximal loxodromic subgroup containing $H$.   
\end{rem}

\subsection{Cone-off space.}\label{S:cone-off}
Let $\rho>0$. Let $\scrY$ be the collection of cylinders $Y=C_H$ such that $(H,Y)\in \scrQ$. The \emph{cone of radius $\rho$ over $Y\in \scrY$}, denoted by $Z_{\rho}(Y)$, is the quotient of $Y\times[0,\rho]$ by the equivalence relation that identifies all points $(y,0)$, for $y\in Y$. The \emph{apex of  the cone $Z_{\rho}(Y)$} is the equivalence class of $(y,0)$. We abuse notation and write $(y,p)$ for the equivalence class of $(y,p)$. We denote by $\scrV$ the collection of apices of the cones over the elements of $\scrY$. Let $\iota\colon Y\into Z_{\rho}(Y)$ be the map that sends $y$ to $(y,\rho)$. The \emph{cone-off space of radius $\rho$ over $X$ relative to $\scrQ$}, denoted by $\dot{X}_{\rho}=\dot{X}_{\rho}(\scrQ,X)$, is obtained by attaching for every $Y\in \scrY$, the cone $Z_{\rho}(Y)$ on $X$ along $Y$ according to $\iota \colon Y\into Z_{\rho}(Y)$. 
The cone off $\dot{X}_{\rho}(\scrQ)$ is endowed with its natural metric, namely, the largest metric bounded by the metrics on $X$ and the cones. In particular, the metric on $\dot X_\rho$ restricted to the image of $X$ is bounded by the metric on $X$. Moreover, there is an action by isometries of $G$ on $\dot{X}_{\rho}$, \cite[Section 5.1]{coulon_geometry_2014}.

\subsection{Quotient space.}
The \emph{quotient space of radius $\rho$ over $X$ relative to $\scrQ$}, denoted by $\q{X}_{\rho}=\q{X}_{\rho}(\scrQ,X)$, is the orbit space $\dot{X}_{\rho}/K$. We denote by $\zeta\colon \cX_{\rho} \onto \q X_{\rho}$ the natural projection and write $\q x$ for $\zeta(x)$ for short. Furthermore, we denote by $\q \scrV$ the image in $\q X_{\rho}$ of the apices $\scrV$. We consider $\q X_{\rho}$ as a metric space equipped with the quotient metric, that is for every $x,x' \in \cX_{\rho}$ 
$$ |\q x- \q x'|_{\q X} = \inf_{h\in K} |hx-x'|_{\cX}.$$
We note that the action of $G$ on $\cX_{\rho}$ induces an action by isometries of $\q G$ on $\q X_{\rho}$. 

Finally, we note that $\dot{X}_{\rho}$ and $\overline X_{\rho}$ are length spaces.

\subsection{Small cancellation groups}

The following lemma summarises Proposition~3.15 and Proposition 6.7 of \cite{coulon_geometry_2014}, see also  \cite[Proposition 5.4, Theorem 5.2]{coulon_partial_2016}.   
It is fundamental in the proof of Theorem \ref{IT:main}.

\begin{lem}[Small Cancellation Theorem \cite{coulon_geometry_2014}]
	\label{lem:small-cancellation-theorem} 
 	There exist positive numbers $\delta_0$, $\q \delta$, $\Delta_0$, $\rho_0$ satisfying the following. Let $0<\delta\le \delta_0$ and $\rho>\rho_0$. Let $G$ be a group acting by isometries on a $\delta$-hyperbolic length  space $X$. Let $\scrQ$ be a loxodromic moving family such that  $\Delta(\scrQ,X) \leq \Delta_0$ and $T(\scrQ,X) > 100\pi \sinh\rho$. Then:
	\begin{enumerate}[label=(\roman*)]
		\item $\q X_{\rho}$ is a $\q \delta$-hyperbolic length space on which $\overline G$ acts by isometries.
		\item Let $r\in (0,\rho/20]$. If for all $v\in \scrV$, the distance $|x-v|\geqslant 2r$ then the projection $\zeta \colon \cX_{\rho} \to \qX_{\rho}$ induces an isometry from $B(x,r)$ onto $B(\q x,r)$.
		\item Let $(H,Y) \in \scrQ$. 
		If $v \in \scrV$ stands for the apex of the cone $Z_{\rho}(Y)$, then the natural projection $\pi \colon  G\onto \q G$ induces an isomorphism from $\stab Y/H$ onto $\stab{\q v}$.  \qed
	\end{enumerate}
\end{lem}

\begin{rem} Even if $X$ is geodesic, it may not be the case that $\overline X_\rho$ is, in addition, a geodesic metric space. On the other hand, if the action of $G$ on $X$ is proper and cocompact, then the action of $\overline G$ on $\overline X$ is proper and cocompact. In this case, the Hopf-Rinow theorem implies that $\overline X_\rho$ is a geodesic metric space. 
\end{rem}

\begin{rem}\label{R:deltabar-universal-constant} Proposition 6.7 of \cite{coulon_geometry_2014} states that $\overline X$ is $\overline \delta$-hyperbolic, where $\overline \delta \leqslant 54\cdot 10^4 \bm{\delta}$ and  where $\bm{\delta}$ is the hyperbolicity constant of the hyperbolic plane. As every $\delta$-hyperbolic space is $\delta'$-hyperbolic whenever $\delta'\geqslant\delta$, we can just fix $\overline \delta =54\cdot 10^4 \bm{\delta}$ so that $\overline \delta$ is a universal constant that does not depend on the choice of $\delta \leqslant \delta_0$.  It does also not depend on the length of the relators, that is, $T(\scrQ,X)$.
\end{rem}

\begin{rem}
\label{rem:small-cancellation-theorem}
The constants $\delta_0$, $\q \delta$, $\Delta_0$, $\rho_0$ are independent of $G$, $X$, $\scrQ$ or $\delta$. Moreover $\delta_0$ and $\Delta_0$ (respectively $\rho_0$) can be chosen arbitrarily small (respectively large). We refer to $\delta_0$, $\q \delta$, $\Delta_0$, $\rho_0$ as \emph{the constants of the Small Cancellation Theorem}. 
\end{rem}
 Due to a standard rescaling argument, see for instance \cite{coulon_theorie_2016}, we can assume that $G$, $X$, $\scrQ$ and $\delta$ satisfy the assumptions of Lemma \ref{lem:small-cancellation-theorem}. This is explained in the next remark. 
\begin{rem}\label{R:rescaling}  Let $\rho \geqslant \rho_0$. 
Let $G$ be a group acting $(\kappa,N)$-acylindrically on a $\delta$-hyperbolic space $X$. Let $\scrQ$ be a loxodromic moving family satisfying the geometric $C''(\lambda,\mu)$-small cancellation condition for the action of $G$ on $X$, 
where 
$$\lambda\le \frac{\Delta_0}{100\pi\sinh\rho}\hbox{ and } \mu\ge \frac{100\pi \sinh \rho}{\delta_0}\cdot \frac{\kappa}{\delta}.$$ We define a rescaling parameter
$$\sigma=\min\qty{\frac{\delta_0}{\kappa},\frac{\Delta_0}{\Delta(\scrQ,X)}}.$$
Note that $\mathcal{X}$ is $\sigma\delta$-hyperbolic and the action of $G$ on $\mathcal{X}$ is $(\sigma\kappa,N)$-acylindrical. Also, 
$$\sigma\delta\le \sigma\kappa \le \delta_0,$$
by the standing assumption that $\kappa\ge \delta$. In particular, the action of $G$ on $\mathcal{X}$ is $(\delta_0,N)$-acylindrical. In addition, we have
\begin{align*}
	\Delta(\scrQ,\mathcal{X}) &\le \sigma \Delta(\scrQ,X) \le \Delta_0,\\ T(\scrQ,\mathcal{X}) &\ge \sigma T(\scrQ,X)\ge \sigma \max\qty { \mu\delta, \frac{\Delta(\scrQ,X)}{\lambda}} \ge 100\pi\sinh\rho.
\end{align*}
 Thus $G$, $\mathcal{X}$ and $\scrQ$ satisfy the hypothesis of the Small Cancellation Theorem (Lemma \ref{lem:small-cancellation-theorem}).
\end{rem}

From now on we assume that  $\delta\le \delta_0$, $\rho\geqslant \rho_0$, $G$, $X$, and $\scrQ$ satisfy the hypothesis of the Small Cancellation Theorem (Lemma \ref{lem:small-cancellation-theorem}). We also assume that the action of $G$ on $X$ is $(\kappa,N)$-acylindrical, with $\kappa = \delta_0$. This is justified by Remark \ref{R:rescaling}. 

We now have the following useful results at our disposal. 

\begin{lem}[\hspace{-0.15mm}{\cite[Proposition~5.16]{coulon_partial_2016}}]
	\label{lem:quotient-elementary}
	Let $E$ be an elliptic (respectively loxodromic) subgroup of $G$ acting on $X$. Then the image of $E$ under the natural projection $\pi \colon  G\onto \q G$ is elliptic (respectively elementary) for its action on $\q X_{\rho}$.
\end{lem}

\begin{lem}[\hspace{-0.15mm}{\cite[Proposition~5.17]{coulon_partial_2016}}]
	\label{lem:quotient-elliptic}
	Let $E$ be an elliptic subgroup of $G$ acting on $X$. Then the natural projection $\pi \colon  G\onto \q G$ induces an isomorphism from $E$ onto its image.
\end{lem}

\begin{lem}[\hspace{-0.15mm}{\cite[Proposition~5.18]{coulon_partial_2016}}]
	\label{lem:lifting-elliptic-subgroups}
	Let $\q E$ be an elliptic subgroup of $\q G$ for its action on $\q X_{\rho}$.
	One of the following holds.
	\begin{enumerate}[label=(\roman*)]
		\item  There exists an elliptic subgroup $E$ of $G$ for its action on $X$ such that the natural projection $\pi\colon G \onto \q G$ induces an isomorphism from $E$ onto $\q E$.
		\item There exists $\q v \in \q{\scrV}$ such that $\bar E \subset \stab {\q v}$.
	\end{enumerate}
\end{lem}

\begin{lem}[\hspace{-0.15mm}{\cite[Proposition~9.13]{coulon_product_2022}}]
	\label{lem:energy-lift} 
	Let $\q U \subset \qG$ be a finite set such that $L(\q U)\leqslant \rho/5$. If, for every $\q v\in \q {\scrV}$, the set $\q U$ is not contained in $\stab{\q v}$, then there exists a pre-image $U\subset G $ of $\q U$ of energy $L(U) \leqslant \pi \sinh L(\q U).$
\end{lem}

\begin{lem}[Greendlinger's Lemma, {\cite[Theorem~3.5]{coulon_detecting_2018}}]
	\label{lem:greendlingers-lemma}
	Let $x\in X$. Let $g\in G$. If $g\in K-\{1\}$, then there is $(H,Y)\in \scrQ$ with the following property. Let $y_0$ an $y_1$ be the respective projections of $x$ and $gx$ on $Y$. Then 
	$$|y_0-y_1|>T(H,X)-2\pi\sinh{\rho}-23\delta.$$
\end{lem}

\begin{rem}
	Lemma \ref{lem:greendlingers-lemma} is obtained from \cite[Theorem~3.5]{coulon_detecting_2018} after applying \cite[Proposition~1.11]{coulon_detecting_2018}, \cite[Proposition~2.4 (2)]{coulon_geometry_2014} and \cite[Lemma~2.31]{coulon_geometry_2014}. In \cite[Theorem~3.5]{coulon_detecting_2018} there is the additional assumption that the loxodromic moving family is finite up to 
	conjugacy. This assumption is only needed to make sure that the action is co-compact, hence, that the quotient group is hyperbolic. We do not need it here.
\end{rem}

\begin{lem}[\hspace{-0.15mm}{\cite[Proposition~5.33]{dahmani_hyperbolically_2017}}]
\label{lem:small-cancellation-acylindricity}
If the action of $G$ on $X$ is acylindrical, then so is the action of $\q G$ on $\q X_{\rho}$.
\end{lem}
\begin{rem}\label{rem:small-cancellation-acylindricity} If all elementary subgroups of $G$ are cyclic and the action on $X$ is $(\kappa,N)$-acylindrical, then the action of $\q G$ on $\q X_{\rho}$ is $(\overline{\kappa},\overline{N})$-acylindrical, where $\overline{\kappa}$ can be explicitly computed in $\kappa$ and $\rho$, and $\overline{N}$ can be explicitly computed in $N$, $\kappa$ and $\rho$. See \cite[Proof of Proposition~5.33]{dahmani_hyperbolically_2017} and \cite[Proposition 9.9]{coulon_product_2022}.
\end{rem}

\section{Shortening and shortening-free words}

\label{sec:shortening-shortening-free-words}

Let $\delta_0 >0$, $\Delta_0 >0$ and $L_0>0$, $\delta<\delta_0  \hbox{ and }  \alpha= 200\delta.$ 

\subsection{Shortening words}
\label{subsec:shortening-words}

Let $U\subset G$ be an $\alpha$-reduced subset at a point $p\in X$ (Definition \ref{df:reduced-subset}). Recall that $\scrQ$ is a loxodromic moving family. We assume that
$$0<L(U,p)\le L_0, \quad \text{and} \quad \Delta(\scrQ,X)\le \Delta_0.$$

Given a reduced word $w\equiv u_1\cdots u_n \in \F(U)$ and $(H,Y)\in \scrQ$, we fix 
$$x_0=p,\quad x_1=u_1p,\quad x_2=u_1u_2p,\quad \cdots,\quad x_n=u_1\cdots u_np,$$
and let $y_i$ be a projection of $x_i$ on $Y$, for every $i\in\zinterval{0}{n}$.

\begin{df}[\emph{Shortening word}]
	\label{df:shortening} Let $\tau\geq \Delta_0+2L_0+223\delta$ and let $w\equiv u_1\cdots u_n \in \F(U)$. Let $(H,Y)\in \scrQ$. We say that $w$ is a $\tau$-\emph{shortening word over} $(H,Y)$ if $w$ is a reduced word and if the following holds: 
	\begin{enumerate}[label=(S\arabic*)]
		\item $|y_0-y_n|>\tau$,
		\item $|x_0-y_0|\leqslant\frac{1}{2}|u_1p-p|-\alpha$, and $|x_n-y_n|\leqslant\frac{1}{2}|u_np-p|-\alpha.$
	\end{enumerate}
	A \emph{minimal $\tau$-shortening word over} $(H,Y)$ is a $\tau$-shortening word over $(H,Y)$ none of whose proper prefixes are $\tau$-shortening words over $(H,Y)$. 
\end{df}

\begin{rem}
\label{rem:non-identity}
By the triangle inequality and the choice of $\tau_0$, $|x_0-x_n|\ge |y_0-y_n|-|x_0-y_0|-|x_n-y_n|>0.$  In particular, $\tau$-shortening words are always non-trivial.
\end{rem}

We now collect some properties of $\tau$-shortening words that we later use in our counting argument, see Proposition \ref{prop:growth shortening-free} below. 

By the next proposition, condition (S2) of Definition \ref{df:shortening} is closed under taking subwords. 

\begin{prop}

\label{prop:shortening-subword}	
	
Let $w\equiv u_1\cdots u_n$ be a $\tau$-shortening word over $(H,Y)\in\scrQ$. Then,
$$|x_i-y_i|<\frac{1}{2}\min\{|u_ip-p|,|u_{i+1}p-p|\}-\alpha, \hbox{ for every $i\in\zinterval{1}{n-1}.$}$$ 

\end{prop}

\begin{proof}
Let $i\in\zinterval{1}{n-1}$. Let $z_i$ be a projection of $x_i$ on $[y_0,y_n]$. Since $Y$ is $10\delta$-quasi-convex (Lemma \ref{L: invariant cylinder}), there exist $z_i'\in Y$ such that $|z_i-z_i'|\le 11\delta$. By definition,
$|x_i-y_i|\le d(x_i,Y)+\delta \le |x_i-z_i'|+\delta.$
Thus, by the triangle inequality,
$$|x_i-y_i|\le |x_i-z_i|+|z_i-z_i'|+\delta\le |x_i-z_i|+12\delta.$$
By definition and Lemma \ref{lem:projections-gromov-product}, 
$$|x_i-z_i|\le d(x_i,[y_0,y_n]) \le (y_0,y_n)_{x_i}+4\delta.$$
We claim that $(y_0,y_n)_{x_i}\le (x_0,x_n)_{x_i}+2\delta$. By hyperbolicity,
$$
\min\{(x_0,y_0)_{x_i},(y_0,y_n)_{x_i},(y_n,x_n)_{x_i}\}\le (x_0,x_n)_{x_i}+2\delta.
$$
The minimum is attained by $(y_0,y_n)_{x_i}$. Indeed, by the triangle inequality, Proposition~\ref{prop:broken geodesic reduced subset}(ii)  and assumption (S2) of Definition \ref{df:shortening}, 
$$(x_0,y_0)_{x_i}\ge |x_0-x_i| - |x_0-y_0| > \frac{1}{2}|u_{i}p-p|+ 2(i-1)(\alpha+\delta)+\alpha.$$ The estimate for $(y_n,x_n)_{x_i}$ is similar. On the other hand,   by Proposition \ref{prop:broken geodesic reduced subset}(iii), { $$(x_0,x_n)_{x_i} < \frac{1}{2}\min \qty {|u_ip-p|, |u_{i+1}p-p|}-\alpha-48\delta.$$ }
Combining the estimates then yields the assertion of the lemma. 
\end{proof}

\begin{prop}
	
	\label{prop:shortening-bounds}
	
	Let $w\equiv u_1\cdots u_n$ be a $\tau$-shortening word over $(H,Y)\in\scrQ$. 
	\begin{enumerate}[label=(\roman*)]
		\item We have
		$$|w|_U\ge \frac{\tau-50\delta}{L(U,p)}.$$
		\item If $w$ is a minimal $\tau$-shortening word over $(H,Y)$, then
		$$|w|_U\le \frac{\tau}{\alpha}+1.$$
	\end{enumerate}
\end{prop}

\begin{proof} We first prove (i). By the triangle inequality,
		$|x_0-x_n|\le L(U,p) |w|_U.$
		 Since $Y$ is $10\delta$-quasi-convex and $|y_0-y_n|\ge \tau \ge 23\delta$ (see \cite[Chapitre 2, Proposition 2.1]{coornaert_geometrie_1990})
		$$|x_0-x_n|\ge |x_0-y_0|+|y_0-y_n|+|y_n-x_n|-46\delta>\tau-50\delta.$$
	This proves (i), and we turn to (ii). 	
Let $w$ be a minimal $\tau$-shortening word over $(H,Y)$. Let $w'\equiv u_1\cdots u_{n-1}$. By definition,
		$|w|_U=|w'|_U+1.$
		By Proposition \ref{prop:broken geodesic reduced subset} (ii), 
		$$
		|w'x_0-x_0|\ge \frac{1}{2}|u_1p-p|+\frac{1}{2}|u_{n-1}p-p|+\alpha(|w'|_U-1).
		$$
		By the triangle inequality,
		$$		|w'x_0-x_0| \le |x_{n-1}-y_{n-1}|+|y_{n-1}-y_0|+|y_0-x_0|. 
		$$		
		By Property (S2), $
		|x_0-y_0| < \frac{1}{2} |u_1p-p|-\alpha.
		$
		By Proposition \ref{prop:shortening-subword}, 
		$
		|x_{n-1}-y_{n-1}| < \frac{1}{2} |u_{n-1}p-p|-\alpha.
		$
		As $w'$ is not a $\tau$-shortening over $(H,Y)$, we conclude that $|y_{n-1}-y_0| \le \tau$. Consequently, $|w'|_U\le \frac{\tau}{\alpha}$. Thus, $|w|_U\le \frac{\tau}{\alpha}+1$.
\end{proof}

\begin{prop}
	
	\label{prop:shortening-fellow-travelling}
	
	Let $(H_1,Y_1)$, $(H_2,Y_2)\in\scrQ$. Let $w\in \F(U)$. If $w$ is a $\tau$-shortening word over both $(H_1,Y_1)$ and $(H_2,Y_2)$, then $(H_1,Y_1)=(H_2,Y_2)$.
\end{prop}

\begin{proof}
	Let $w$ be a $\tau$-shortening word over $(H_1,Y_1)$ and $(H_2,Y_2)$. By assumptions, it is enough to show that $\diam(Y_1^{+20\delta}\cap Y_2^{+20\delta})>\Delta(\scrQ,X)$. By quasi-convexity, see \cite[Lemma 2.2.2]{delzant_courbure_2008},
	$$\diam(Y_1^{+20\delta}\cap Y_2^{+20\delta})\ge \diam(Y_1^{+L_0}\cap Y_2^{+L_0})-2L_0-4\delta.$$	
	Let $i\in\zinterval{1}{2}$. Let $s_i$ and $t_i$ be respective projections of $p$ and $wp$ on $Y_i$. We claim that $s_1,t_1\in Y_1^{+L_0}\cap Y_2^{+L_0}$. Indeed, by condition (S2) for $w$, we have that 
	$\max\{|p-s_i|,|wp-t_i|\}\le L_0/2,$ hence, the claim. 
		Thus,
	$\diam(Y_1^{+L_0}\cap Y_2^{+L_0})\ge |s_1-t_1|>\tau.$
By choice of $\tau$, $\diam(Y_1^{+20\delta}\cap Y_2^{+20\delta})>\Delta(\scrQ,X)$.
\end{proof}

\begin{prop}
	\label{prop:shortening-proper-prefix}
	For every $(H,Y)\in \scrQ$, there are at most two minimal $\tau$-shortening words over $(H,Y)$.
\end{prop}
Proposition \ref{prop:shortening-proper-prefix} should be compared to Proposition 3.8 of \cite{coulon_product_2022}. The proof of these two statements relies on the geodesic extension property (in our setting Proposition \ref{prop:geodesic extension property} above). For convenience of the reader we translate the argument of \cite{coulon_product_2022} to our setting.  
\begin{proof}  
	Let $(H,Y)\in \scrQ$. Let $\partial H =\{\eta^-,\eta^+\}$ and let $\gamma\colon \R\to X$ be a $10^3\delta$-local $(1,\delta)$-quasi-geodesic from $\eta^-$ to $\eta^+$. Let $q$ be a projection of $p$ on $\gamma$. We assume that $q=\gamma(0)$. Let $\scrS_{(H,Y)}$ be the set of $\tau$-shortening words in $\F (U)$ over $(H,Y)$. Assume that $\scrS_{(H,Y)}$ is non-empty, otherwise the statement is true. Let $\scrS_{(H,Y)}^+\subseteq \scrS_{(H,Y)}$ (respectively, $\scrS_{(H,Y)}^- \subseteq \scrS_{(H,Y)})$ be the set of words $w\in \scrS_{(H,Y)}$ such that $wp$ has a projection $\gamma(t)$ on $\gamma$ with $t\ge 0$ (respectively, $t\le0$). We note that $\scrS_{(H,Y)}=\scrS_{(H,Y)}^+\bigcup \scrS_{(H,Y)}^-$.
	
	Let $w_1,w_2\in \scrS_{(H,Y)}^+$. Without loss of generality, the projection of $w_1p$ on $\gamma$ is closer to $0$ than the projection of $w_2p$, that is,  $0\le t_1\le t_2$. We prove that $(p,w_2p)_{w_1p}< \frac{1}{2}|u_mp-p|.$ Proposition \ref{prop:geodesic extension property} then implies that $w_1$ is a prefix of $w_2$. Thus, there is at most one minimal $\tau$-shortening word over $(H,Y)$ in $\scrS_{(H,Y)}^+$. By symmetry, there is at most one minimal $\tau$-shortening word over $(H,Y)$ in $\scrS_{(H,Y)}^-$, which yields the assertion. 
	 
	 Let $q_1=\gamma(t_1)$ and $q_2=\gamma(t_2)$ be the respective projections of $w_1p$ and $w_2p$ on $\gamma$. 
	 By the triangle inequality,
		\begin{equation}
				\label{eqn:shortening-2-1}
				(p,w_2p)_{w_1p}\le |w_1p-q_1|+(w_2p,p)_{q_1}.
		\end{equation}
		Assume that $w_1\equiv u_1\cdots u_m$.  

We  note that $|w_1p-q_1|<\frac{1}{2}|u_mp-p|-\alpha+100\delta$. Indeed, by definition, $\gamma\in Y$. Consequently, $|w_1p-q_1|\le d(w_1p,Y)+100\delta$.
		Property (S2) of $w_1$ then implies the claim. 
				
Next, we prove that $(w_2p,p)_{q_1}\le 29\delta$. Indeed, by definition 
		$$(w_2p,p)_{q_1}=\frac{1}{2}(|w_2p-q_1|+|p-q_1|-|w_2p-p|).$$
		Since $w_2$ is a $\tau$-shortening word over $(H,Y)$, the property (S1) implies 
		$|q_2-q|>\tau\ge 23\delta.$
		By $10\delta$-quasi-convexity of $Y$, \cite[Chapitre 2, Proposition 2.1]{coornaert_geometrie_1990},
		$$|w_2p-p|\ge |w_2p-q_2|+|q_2-q|+|q-p|-46\delta.$$
As $(q_2,q)_{q_1}\le 6\delta	$ (Lemma \ref{lem:quasi-line} (i)),
		$$|q_2-q|= |q_2-q_1|+|q_1-q|-(q_2,q)_{q_1}\ge |q_2-q_1|+|q_1-q|-12 \delta.$$
		Note that here we have used that $0\le t_1\le t_2$. By the triangle inequality,
		$
		|w_2p-q_2|\ge |w_2p-q_1|-|q_2-q_1|,
		$
so that 
		$|w_2p-p|\ge |w_2p-q_1|+|q_1-p|-58\delta.$
		Consequently,
$(w_2p,p)_{q_1}\le 29\delta,$ as claimed.

Plugging the two estimates into \eqref{eqn:shortening-2-1} finishes the proof. 
\end{proof}

\subsection{The growth of shortening-free words}
\label{subsec:counting}
Here we count shortening-free words. 

\begin{df}[\emph{Shortening-free word}]
	
	Let $w\equiv u_1\cdots u_n\in \F(U)$ be a reduced word. Let $(H,Y)\in \scrQ$. We say that $w$ \emph{contains a} $\tau$-\emph{shortening word over} $(H,Y)$ if $w$ splits as $w\equiv w_0w_1w_2$, where $w_1$ is a $\tau$-shortening word over $(H,Y)$. We say that $w$ is a $\tau$-\emph{shortening-free word} if for every $(H,Y)\in\scrQ$, the word $w$ does not contain any $\tau$-shortening word over $(H,Y)$. We denote by $F(\tau)\subset \F(U)$ the subset of $\tau$-shortening-free words. 
\end{df}

We first discuss the counting argument to estimate the number of shortening-free words. Then we show that shortening-free words embed into $\overline{G}$, see Proposition \ref{prop:embedding shortening-free} below. 

Recall that the natural homomorphism $\F(U)\to G$ is injective (Proposition \ref{prop:reduced subset qi embedding}). Also, we identify the elements of $\F(U)$ with their images in $G$. Let $B_U(n)\subset \F(U)$ be the ball of radius $|w|_U\le n$, where $n \ge 0$. Note that $B_U(n)=(U \sqcup U^{-1}\sqcup\{1\})^n$ when $n \ge 1$.

\begin{prop}
	\label{prop:growth shortening-free}
	Let $U$ be an $\alpha$-reduced subset of at least two elements such that $L(U)\leq L_0$. There is $\tau_1$ depending on $L_0$, $\Delta_0$ and $\delta$ such that, for all $\tau\ge \tau_1$ and for all $n\ge 0$, 
	$$|F(\tau)\cap B_U(n+1)|\ge |U||F(\tau)\cap B_U(n)|.$$
	In particular, $|F(\tau)\cap B_U(n)|\ge |U|^n .$
\end{prop}

Proposition \ref{prop:growth shortening-free} brings Proposition 3.11 of \cite{coulon_product_2022} for the growth of aperiodic elements in strongly reduced sets to our setting of shortening-free words in $\alpha$-reduced subsets. 

\begin{rem} As a semi-group $\F(U)$ is generated by $V=U\sqcup U^{-1}$. As $|V|=2|U|$, the estimate of Proposition \ref{prop:growth shortening-free} is as expected from \cite[Proposition 3.11]{coulon_product_2022}.
\end{rem}

The proof of Proposition \ref{prop:growth shortening-free} is as the proof of Proposition 3.11 in \cite{coulon_product_2022}, where \cite[Proposition 3.6]{coulon_product_2022}  and \cite[Proposition 3.8]{coulon_product_2022} are replaced by our Proposition \ref{prop:shortening-fellow-travelling} and Proposition \ref{prop:shortening-proper-prefix} respectively. Essentially, all that is left to do is to adapt notation.
 We let 
$$Z=\setc{w\in \F(U)}{w\equiv w_0u\text{ as a reduced word}, w_0\in F(\tau),u\in U\sqcup U^{-1}}.$$
For every $(H,Y)\in \scrQ$, we denote by $Z_{(H,Y)}\subset Z$ the set of elements $w\in Z$ that split as $w\equiv w_1w_2$, where $w_1\in F(\tau)$ and $w_2$ is a $\tau$-shortening word over $(H,Y)$.

\begin{lem}
	\label{lem:growth shortening-free 1}	
	The set $Z$ is contained in the disjoint union of $F(\tau)$ and $\bigcup_{(H,Y)\in\scrQ}Z_{(H,Y)}.$
\end{lem}

\begin{proof}
	By definition, $F(\tau)$ and $\bigcup_{(H,Y)\in\scrQ}Z_{(H,Y)}$ are disjoint. Let $w\in Z-F(\tau)$. Then there is $w_0\in F(\tau)$ and $u\in U\sqcup U^{-1}$ such that $w\equiv w_0u$. Moreover, $w$ contains $\tau$-shortening word over some $(H,Y)\in \scrQ$ that we call $w_2$. By definition, $w_2$ cannot be a subword of $w_0$. Hence, $w_2$ is a suffix of $w$. Therefore, $w\in Z_{(H,Y)}$.
\end{proof}

Let $n\ge 0$. Lemma \ref{lem:growth shortening-free 1} implies that
\begin{equation}
	\label{eq:growth shortening-free}
	|F(\tau)\cap B_U(n)|\ge |Z\cap B_U(n)|-\sum_{(H,Y)\in\scrQ}|Z_{(H,Y)}\cap B_U(n)|.
\end{equation}

The next step is to estimate each term in the right side of the above inequality. The following lemma is a direct consequence of the definition of $Z$.

\begin{lem}
	\label{lem:growth shortening-free 2} 	For every $n\ge 0$, 
$|Z\cap B_U(n+1)|= (2|U|)|F(\tau)\cap B_U(n)|.$
\end{lem}
We now let  
$\tau_0=\Delta_0+2L_0+223\delta$, $\tau>\tau_0$,  
	$ b=\left\lceil \frac{\tau_0}{200\delta}+2\right\rceil+1,$ and $M=\left\lfloor\frac{\tau_0-50\delta}{L_0}\right\rfloor.$

\begin{lem}
	\label{lem:growth shortening-free 3}
	For every $n\ge 0$,
	$$\sum_{(H,Y)\in\scrQ}|Z_{(H,Y)}\cap B_U(n)|\le 2(2|U|)^{b}|F(\tau)\cap B_U(n-M)|.$$
\end{lem}
\begin{proof}
	Let $\scrQ_0$ be the set of $(H,Y)\in\scrQ$ for which there is a $\tau$-shortening word in $\F(U)$ over $(H,Y)$. We have,
	$$\sum_{(H,Y)\in\scrQ}|Z_{(H,Y)}\cap B_U(n)|=\sum_{(H,Y)\in\scrQ_0}|Z_{(H,Y)}\cap B_U(n)|.$$
	
	\begin{cla}
		$|Z_{(H,Y)} \cap B_U(n)|\le 2|F(\tau) \cap B_U(n-M)|$, for every $(H,Y)\in \scrQ_0$.
	\end{cla}

	\begin{proof}
		Let $(H,Y)\in \scrQ_0$. Let $w\in Z_{(H,Y)} \cap B_U(n)$. Since $w\in Z_{(H,Y)}$, there are $w_1\in F(\tau)$ and a $\tau$-shortening word $w_2$ over $(H,Y)$ such that $w\equiv w_1w_2$. In particular, $|w_1|_U=|w|_U-|w_2|_U.$
		By Proposition \ref{prop:shortening-bounds} (i), 
		$$|w_2|_U\ge \frac{\tau-50\delta}{L_0}\ge M> 0.$$
		Therefore, $w_1\in F(\tau)\cap B_U(n-M)$. As $w\in Z$, no proper prefix of $w_2$ is a $\tau$-shortening word. By Proposition \ref{prop:shortening-proper-prefix} there are at most $2$ possible choices for $w_2$. In total, there are at most $2|F(\tau)\cap B_U(n-M)|$ choices for $w$. This proves our claim.
	\end{proof}
	
	\begin{cla}
		$|\scrQ_0|\le (2|U|)^{b}$
	\end{cla}

	\begin{proof}
As $|U|\ge 2$, $|B_U(b)|\le (2|U|)^{b}.$
		Consequently, it suffices to show that there exists an injective map $\chi\colon \scrQ_0\to B_U(b)$. Let $(H,Y)\in\scrQ_0$, and let $w$ be a $\tau$-shortening word over $(H,Y)$. 
		As $\tau\ge \tau_0$,  $w$ is a $\tau_0$-shortening word over $(H,Y)$ as well. Let $w'$ be the shortest prefix of $w$ that is a $\tau_0$-shortening word over $(H,Y)$. In particular, $w'$ is a minimal $\tau_0$-shortening word over $(H,Y)$. We define $\chi(H,Y)=w'$. Since $\alpha= 200\delta$, by Proposition \ref{prop:shortening-bounds} (ii), $|w'|_U\le b$.
		By Proposition \ref{prop:shortening-fellow-travelling}, there is at most one $(H,Y)\in\scrQ$ such that $w'$ is a $\tau_0$-shortening word over $(H,Y)$. Hence $\chi$ is well-defined and injective. 
	\end{proof}
	The desired estimation is obtained from the two claims above.	
\end{proof}

Let $\varepsilon\in (0,3/4)$ be an auxiliary parameter. Given $\varepsilon$, we fix
$$\nu=(1-\varepsilon)2|U|,\quad \xi=2(2|U|)^b,\quad\text{and }  \sigma=\frac{\varepsilon}{2(1-\varepsilon)\xi}.$$

\begin{lem}
	\label{lem:computation}
	For every $\varepsilon\in (0,3/4)$ and $b>1$, there is $M_0\ge 0$ with the following property. 
	If $|U|\ge 2$, then for every $M\ge M_0$, we have 
	$$\frac{1}{\nu^{M}}\le \sigma.$$
\end{lem}

\begin{proof}
	Let $M\ge b$. A computation yields 
	$$\log\qty(\frac{1}{\sigma\nu^M})\le (b-M)\log(2|U|)-M\log(1-\varepsilon)-\log(\frac{\varepsilon}{4(1-\varepsilon)}).$$
	Since $M\ge b$ and $|U|\ge 2$, we have $(b-M)\log(2|U|)\le (b-M)\log 4.$
	Therefore,
	$$\log\qty(\frac{1}{\sigma\nu^M})\le -M[\log 4+\log(1-\varepsilon)]+b\log4-\log(\frac{\varepsilon}{4(1-\varepsilon)}).$$
	We put
	$$
	d_1=b\log4-\log(\frac{\varepsilon}{1-\varepsilon}),\quad		d_2=\log 4+\log(1-\varepsilon). 
	$$
	As $b\ge 1$ and $\varepsilon\in (0,3/4)$, $d_1>0$  and $d_2>0$. If $M\ge \frac{d_1}{d_2}$, then $\log\qty(\frac{1}{\sigma\nu^M})\le 0$. 
\end{proof}

\begin{proof}[Proof of Proposition \ref{prop:growth shortening-free}]
	Let $M_0\ge 0$ be given by Lemma \ref{lem:computation}. We let  
	$$\tau_1 = \max\{\tau_0,L_0(M_0+1)+50\delta_0\}$$
	and note that $\tau_1$ depends only on $L_0$, $\varepsilon$, $b$ and $\delta_0$. 
	Recall that $|U|\ge 2$, that $\tau\ge \tau_1$ and that
	$$\nu=(1-\varepsilon)(2|U|),\quad \xi=2(2|U|-1)^b,\quad\sigma=\frac{\varepsilon}{2\xi(1-\varepsilon)},\quad\text{and } M=\left\lfloor\frac{\tau-50\delta_0}{L_0}\right\rfloor.$$
	In particular, $M\ge M_0$. For every $n\ge 0$, we let
	$$c(n)=|F(\tau)\cap B_U(n)|.$$
	We must prove that for every $n\ge 1$, $c(n)\ge \nu c(n-1).$ 
We prove this by induction.

	First note that $c(1)\ge \nu$. Indeed, as $B_U(1)>2|U|$, it suffices to show that $B_U(1)\subset F(\tau)$.  Let $w\in F(\tau)$ then, by  Proposition \ref{prop:shortening-bounds} (i), $|w|_U\ge \frac{\tau-50\delta_0}{L_0}$. By choice of $\tau_0$, $|w|_U>1$. This proves our claim.
	
Next let $n\ge 1$ and assume that $c(m)\ge \nu c(m-1)$, for every $m\in \zinterval{1}{n}$. We claim that $c(n+1)\ge \nu c(n)$. Indeed, by  \eqref{eq:growth shortening-free},
	$$c(n+1)\ge |Z\cap B_U(n+1)|-\sum_{(H,Y)\in\scrQ}|Z_{(H,Y)}\cap B_U(n+1)|.$$
	It follows from Lemma \ref{lem:growth shortening-free 2} and Lemma \ref{lem:growth shortening-free 3} that
	$$c(n+1)\ge 2|U|c(n)-\xi c(n+1-M).$$
	The induction hypothesis implies that $c(n-k)\le \nu^{-k}c(n)$, for every $k\ge 0$ . Note that $M-1\ge 0$. Letting $k=M-1$, we obtain
	$$c(n+1)\ge\qty(1-\frac{\xi\nu}{2|U|}\frac{1}{\nu^{M}})(2|U|)c(n).$$
	Recall that $\nu=(1-\varepsilon)2|U|$. In addition, Lemma \ref{lem:computation} yields that $\frac{1}{\nu^{M}}\le \sigma.$
Now
	$$\frac{\xi\nu}{2|U|}\sigma=\frac{\xi(1-\varepsilon)(2|U|)}{2|U|}\frac{\varepsilon}{2\xi(1-\varepsilon)}=\frac{\varepsilon}{2}\le \varepsilon.$$
	This proves our claim. Fixing $\varepsilon=1/2$ implies the proposition.
\end{proof}

\subsection{The injection of shortening-free words}
\label{subsec:embedding}
Finally, we show that shortening-free words inject in $\overline{G}$, see Proposition \ref{prop:embedding shortening-free} below. We first prove a version of \cite[Proposition 4.2]{coulon_product_2022} in our setting of shortening-free words and $\alpha$-reduced sets. 

We recall that $U$ denotes an $\alpha$-reduced subset of $G$. 

\begin{lem}
	\label{lem:barrier implies shortening} Let $\tau\geq \Delta_0+10L_0+240\delta_0$. 
	Let $w\equiv u_1\cdots u_n$ be an element of $\F(U)$ in reduced form.  Let $(H,Y) \in \scrQ$. Let $y_0$ and $y_n$ be respective projections of $p$ and $wp$ on $Y$. If $|y_0-y_n| > \tau$,	then $w$ contains a $(\tau-8L_0-8\delta)$-shortening word over a conjugate of $(H,Y)$.
\end{lem}

The main point is to promote the proof of \cite[Proposition 4.2]{coulon_product_2022} from strongly $\alpha$-reduced sets to $\alpha$-reduced sets. Besides this technical point the argument is the same. 
 
\begin{proof} Let $\gamma_w$ be a geodesic from $p$ to $wp$ and let $z_i$ be a projection of $x_i$ on $\gamma_w$. We first claim that the points $z_0$, $z_1$, $z_2$, $\ldots$, $z_n$ lie in this order on $\gamma_w$ going from $z_0=p$ to $z_n=wp$. Indeed, if $i>0$, by Proposition \ref{prop:broken geodesic reduced subset}(iii) 
\begin{align*}
(x_0,x_i)_{x_{i+1}} = |x_i-x_{i+1}|-(x_0,&x_{i+1})_{x_i}
 \geqslant |p-u_{i+1}p|-(x_{i-1},x_{i+1})_{x_i} -2\delta \\
& > \frac{1}{2}|p-u_{i+1}p|+\alpha \geqslant (x_{i},x_{i+2})_{x_{i+1}} +2\alpha +2\delta.
\end{align*}
If $z_{i+1}$ lies before $z_i$ on $\gamma_w$, then hyperbolicity, Lemma \ref{lem:projections-gromov-product} and Proposition \ref{prop:broken geodesic reduced subset}(iii) implies that 
$$(x_0,x_i)_{x_{i+1}}\leqslant |z_{i+1}-x_{i+1}| +2\delta \leqslant (x_0,x_n)_{x_{i+1}} +6\delta\leqslant (x_{i},x_{i+2})_{x_{i+1}}+8\delta.$$
{Combined with the previous estimate this yields that $\alpha \leqslant 3\delta$, a contradiction.}

Next recall that $y_0$, $y_i$ and $y_n$ are projections of $x_0=p$, $u_1\cdots u_ip=x_i$ and $wp=x_n$, respectively, on $Y$.  Assume that $|y_0-y_n|>\tau$. Since $Y$ is $10\delta$-quasi-convex (Lemma \ref{L: invariant cylinder}) and $\tau \ge 23\delta$,  there are $y_0',y_n' \in \gamma_w$ such that 
	$\max\{|y_0-y_0'|, |y_n-y_n'|\} \le 46\delta.$ 
 Up to permuting $y_0'$ and $y_n'$ we assume that $p$, $y_0'$, $y_n'$ and $wp$ are ordered in this way along $\gamma_w$. In particular, there are $i\le j\le n-1$ such that $y_0'\in [z_i,z_{i+1}]\subset \gamma_w$ and $y_n'\in  [z_j,z_{j+1}]\subset \gamma_w$. Let $w_0\equiv u_1\cdots u_{i+1}$ and take the word $w_1$ such that $w_0w_1\equiv u_1\cdots u_j$. We  prove that $w_1$ is a $(2\tau-2L_0-5\delta)$-shortening word over $(w_0^{-1}Hw_0,w_0^{-1}Y)$. Property (S2) follows from Proposition \ref{prop:broken geodesic reduced subset}(iii). Indeed, this implies that 
	$d(x_{i+2},\gamma_w)\leqslant \frac{1}{2}|u_{i+2}p-p| -\alpha -48\delta,$ hence, that $d(x_{i+2},Y)\leqslant \frac{1}{2}|u_{i+2}p-p| -\alpha.$ The argument for $x_{j}$ is similar. It remains to estimate  $|y_{i+1}-y_j|$. By the triangle inequality,
	\begin{align*}
		|y_{i+1}-y_j|&\ge |y_0-y_n|-|y_0-y_{i+1}|-|y_n-y_j|,\\
		|y_0-y_{i+1}|&\le |y_0-y_0'|+|y_0'-x_{i+1}|+|x_{i+1}-y_{i+1}|,\\
		|y_n-y_j|&\le |y_n-y_n'|+|y_n'-x_j|+|x_j-y_j|.
	\end{align*}
	We first estimate $|y_0'-x_{i+1}|$ and $|y_n' -x_j|$. 
	 As $(z_i,z_{i+1})_{y_0'}=0$, hyperbolicity implies that  
	 $$\min\{(z_i,x_i)_{y_0'}, (x_i,x_{i+1})_{y_0'}, (x_{i+1},z_{i+1})_{y_0'}\}\leqslant 2\delta.$$ 
	By definition of the Gromov product 
\begin{align*}
|y_0'-x_i| & \leqslant 2(z_i,x_i)_{y_0'} +|x_i-z_i|,\\
 |y_0'-x_{i+1}|& \leqslant 2(x_i,x_{i+1})_{y_0'} +|x_i-x_{i+1}|,\\
  |y_0'-x_{i+1}|& \leqslant 2(x_{i+1},z_{i+1})_{y_0'} +|z_{i+1}-x_{i+1}|.   
\end{align*}	
It follows from property (S2) and the triangle inequality that $|y_0'-x_{i+1}|\leqslant 2L(U,p)+4\delta$. Similarly, $|y_n' -x_j|\leqslant 2L(U,p)+4\delta$. 

	It follows from Property (S2) that,
	$$
	\max\{|x_{i+1}-y_{i+1}|,|x_j-y_j|\}\le \frac{1}{2}L(U,p)-\alpha +2\delta \le \frac{1}{2} L_0-\alpha + 2\delta.
	$$
	Combining the previous estimations, $|y_{i+1}-y_j|>\tau-8L_0 -8\delta \ge \Delta_0+2L_0+223 \delta$.	
\end{proof}

Let $\tau_1$ be the constant of Proposition \ref{prop:growth shortening-free} depending on $L_0$, $\Delta_0$ and $\delta\leqslant \delta_0$. Recall that $\tau_1\geq \Delta_0+10L_0+240\delta_0$ and let 
$\rho\ge \max\{\rho_0,\log(2[2\tau_1+23\delta_0]+1)\}.$ 
In addition, we recall that 
$T(\scrQ,X)\ge 100\pi\sinh\rho,$ 
that $K=\normal{H\mid (H,Y)\in \scrQ}$ and that $\q G= G/K$. 

\begin{prop}
	\label{prop:embedding shortening-free}
Let $\tau_2=\tau_1+8L_0+8\delta$. The restriction of the natural homomorphism $\F(U)\to \overline{G}$ to the subset of $\tau_2$-shortening-free words is an injection.
\end{prop}

\begin{proof}
Let $w_1,w_2\in \F(U)$ be two $\tau_2$-shortening-free words such that $w_1w_2\in K$. Assume that $w_1w_2\not =1$ in $G$. By Greendlinger's Lemma (Lemma \ref{lem:greendlingers-lemma}), there is $(H,Y)\in \scrQ$ and projections $y_0$ and $y_2$ on $Y$ of $p$, and $w_1w_2p$ respectively, such that 
	$$
	|y_0-y_2|>T(H,X)-2\pi\sinh\rho-23\delta.
	$$
	By definition, $T(H,X)\ge T(\scrQ,X)\ge 100\pi\sinh\rho$. By choice of $\rho$,  
	$$|y_0-y_2|>2\tau_1.$$
	Let $y_1$ be a projection of $w_1p$ on $Y$. Note that $w_1^{-1}y_1$ and $w_1^{-1}y_2$ are projections of $p$ and $w_2p$ on $w_1^{-1}Y$. Also, $(w_1^{-1}Hw_1,w_1^{-1}{Y})\in \scrQ$. Since $w_1$ and $w_2$ are $\tau_2$-shortening-free words, it follows from Lemma \ref{lem:barrier implies shortening} that $
		\max\qty{|y_0-y_1|,|y_1-y_2|} < \tau_1. 
	$ 	By the triangle inequality,
	$|y_0-y_2|\le  2\tau_1.$  
This is a contradiction, hence, $w_1w_2=1$.
\end{proof}

\section{Growth in small cancellation groups}
\label{sec:growth-small}

The goal of this section is to prove Theorem \ref{IT:main}.
 Recall that the constants of the Small Cancellation Theorem (Lemma \ref{lem:small-cancellation-theorem}) are $\delta_0$, $\q\delta$, $\Delta_0$, $\rho_0$, and that $X$ is $\delta$-hyperbolic, for $\delta\leq \delta_0$.  Moreover, $G$ acts $(\kappa,N)$-acylindrically on $X$, where $\kappa=\delta_0$. 
 
 For convenience, we set $\delta=\delta_0$ and view $X$ as a $\delta_0$-hyperbolic space. Moreover, by Remark \ref{rem:small-cancellation-theorem}, we may assume that $$\delta_0\le\frac{\pi\sinh 10^4\q\delta}{10^4\cdot 200} \hbox{ and } \rho_0>5\cdot 10^4\q\delta.$$
  We fix $$L= 2\pi \sinh 10^4\q\delta+\delta_0$$ and note that $L>4000 \kappa$. Let $U\subset G$ be a finite symmetric subset containing the identity such that $L(U)\le L$, and suppose that $U$ is not contained in an elementary subgroup.  
 Let $C=10^6(N+1)$, $n>C^3$, $S\subset U^{10^7n}$ and $p\in X$ be given by  Proposition \ref{P:Pingpong}, so that  
\begin{enumerate}
\item $S$ is $\alpha$-reduced at $p$, for $\alpha=200\delta$,
\item $|S| \geqslant \frac{1}{C^2}|U^n|$,
\item $L(S,p)\le L(U^n,p)\leqslant 10^{13} C^3 L$.
\end{enumerate}
Moreover, $|S|>C$. Note that $C>0$ depends on $N$ but not on the choice of $U$.  
 
 We let $L_0=10^{13} C^3 L$ and let $\tau_1$  be the constant of Proposition \ref{prop:growth shortening-free} depending on $L_0$, $\Delta_0$ and $\delta_0$. Let
\begin{align}
 \rho = \max\qty {\rho_0,\log(2[2\tau_1+23\delta_0]+1)}, \label{E:choice-of-rho}
 \end{align}
so that Proposition \ref{prop:embedding shortening-free} holds.
Note that $\rho$ depends on $N$ and our constants $\delta_0$ and $\overline \delta$, but not on the choice of $U$ or $S$. 

Recall that  $K=\normal{H\mid (H,Y)\in \scrQ}$ and $\q G= G/K$. Moreover, $\q U$ is the image of a  set $U\subset G$ under the natural projection $\pi\colon G\onto \q G$. The following is our key observation. 

\begin{lem}
\label{lem:small-energy}
Let $U\subset G$ be a finite symmetric subset containing the identity such that $L(U)\le L$. If $U$ is not contained in an elementary subgroup, then 
$$\omega(\q U)\ge \frac{1}{10^8} \omega(U).$$
\end{lem}

\begin{proof}
To estimate $\omega(\q U)$ we proceed as follows.  Let $r \ge 1$. Since $U$ is symmetric and contains the identity,  
$B_S(r)\subset U^{10^7 n r}.$ Let $\tau_2$ be the number given by Proposition~ \ref{prop:embedding shortening-free} and let $F(\tau_2)\subset \F(U)$ be the set of $\tau_2$-shortening-free words. We have that 
$$|\q U^{10^7nr}|\ge |\q{B_S(r)}|\ge |\q{F(\tau_2)\cap B_S(r)}|.$$
 By Propositions \ref{prop:embedding shortening-free} and \ref{prop:growth shortening-free}, 
$$|\q{F(\tau_2)\cap B_S(r)}|=|F(\tau_2)\cap B_S(r)| \ge |S|^r.$$
Combining these estimates, we get
$$\omega(\overline{U})\geqslant \limsup_{r\to \infty} \frac{1}{10^7 nr}\log |\overline{U}^{10^7nr}|\geqslant \frac{1}{10^7n}\log {|S|}.$$

Let $n_0=4\frac{\log(C)}{\omega(U)}$. 
If $n\leqslant n_0$, then as $|S|>C$,
\begin{align*}
\omega(\overline{U})& \geqslant \frac{1}{10^7n} \log {|S|} \geqslant \frac{1}{10^8}  \omega(U).
\end{align*} 
If $n>n_0$, we use that $|S| \geqslant \frac{1}{C^2}|U^n|$. By Fekete's inequality, $|U^n|\geqslant 2^{n\omega(U)}$, hence, $|S| \geqslant \frac{1}{C^2}2^{n\omega(U)}$. Thus  
\begin{align*}
\omega(\overline{U}) 
\geqslant \frac{1}{10^7} \left(\omega(U)-\frac{1}{n}2\log(C) \right) 
\end{align*} 
As $n>n_0$ we get that $\frac{1}{n}2\log(C)<\frac 12 \omega(U)$. We conclude that 
\begin{align*}
\omega(\overline U) \geqslant \frac 1{10^7}\left(\frac 12 \omega(U)\right) \geqslant \frac{1}{10^8} \omega(U).
\end{align*}

Altogether, this shows that $\omega(\q U)\ge \frac{1}{10^8} \omega(U)$. 
\end{proof}

\begin{thm}
\label{IT:main2-body}
Let $\xi>0$. If $G$ has $\xi$-uniform uniform exponential growth, then $\q G$ has $\xi'$-uniform uniform exponential growth where $\xi'\ge  \min\{\xi/10^8,\frac{1}{10^5}\log 2\}$.
\end{thm}

\begin{proof}
		Let $\q U\subset \q G$ be a finite symmetric subset containing the identity and denote $\q \Gamma=\group{\q U}$. Recall that $\scrV$ denotes the set of apices of the cone-off space $\dot{X}_{\rho}$. Also recall that the quotient space $\q {X}_{\rho}$ is $\q\delta$-hyperbolic (Lemma \ref{lem:small-cancellation-theorem} (i)) and the action of $\q \Gamma$ on $\q {X}_{\rho}$ is acylindrical (Lemma \ref{lem:small-cancellation-acylindricity}). There are two cases:
	 
	\textbf{Case 1.} \emph{$\q \Gamma$ is elementary.}
	If $\q \Gamma$ is loxodromic, $\q\Gamma$ is virtually nilpotent. Hence we assume that $\q\Gamma$ is elliptic. 
	If $\q\Gamma \leq \stab{\q v}$, for some $\q v\in \q {\scrV}$, then let $(H,Y)\in\scrQ$ such that $v$ is the apex of the cone $Z(Y)$. The natural projection $\pi\colon G\onto \q G$ induces an isomorphism $\stab{Y}/H \xrightarrow{\sim} \stab{\q v}$ (Lemma \ref{lem:small-cancellation-theorem} (iii)). Since the moving family $\scrQ$ is loxodromic, $H$ has finite index in $\stab{Y}$. Hence $\q \Gamma$ is finite, in particular virtually nilpotent. 
	
	Thus we may assume that $\q\Gamma$ does not stabilise any point $\q v$, for every $\q v\in \q {\scrV}$. Then there is an elliptic subgroup $E\subset G$ such that the natural projection $\pi\colon G\onto \q G$ induces an isomorphism $E\xrightarrow{\sim} \q \Gamma$ (Lemma \ref{lem:lifting-elliptic-subgroups}). Since $G$ has $\xi$-uniform uniform exponential growth, the subgroup $E$ is either virtually nilpotent or has $\xi$-uniform exponential growth. In combination with the isomorphism $E\xrightarrow{\sim} \q \Gamma$, we deduce that $\q \Gamma$ either is virtually nilpotent or has $\xi$-uniform exponential growth.
	
	\textbf{Case 2.} \emph{$\q \Gamma$ is non-elementary.}
	 If  $L(\q U)>10^4\q\delta$, then  $\omega(\q U) \ge \frac{1}{10^5} \log 2$ by Lemma \ref{lem:breuillard-fujiwara}.
	Otherwise, $L(\q U)\le 10^4\q\delta$. As $\q \Gamma$ is not elementary, 		
		$\q U$ is not contained in $\stab{\q v}$, for every $\q v\in \q {\scrV}$. As $10^4\q\delta\le \rho/5$, there exists a pre-image $U\subset G$ of $\q U$ of energy $L(U) \leqslant \pi \sinh 10^4\q\delta$ (Lemma \ref{lem:energy-lift}). Without loss of generality, we may assume that $U$ is symmetric and contains the identity. Since $\q \Gamma$ is non-elementary for the action on $\q {X}_{\rho}$, the subgroup $\Gamma$ is non-elementary for the action on $X$ (Lemma \ref{lem:quotient-elementary}). By Lemma \ref{lem:small-energy}, $\omega(\q U)\ge \frac{1}{10^8}\omega(U)\ge \frac{1}{10^8} \xi$.  

This yields that $\omega(\q U) \geqslant \min \{\frac{1}{10^8} \xi,\frac{1}{10^5} \log 2\}$, as claimed.
			\end{proof}

\begin{thm}
	\label{IT:main3-body}
	 Let $\xi>0$. If $\q G$ has $\xi$-uniform uniform exponential growth, then $G$ has $\xi'$-uniform uniform exponential growth, where $\xi'\ge \min\{\xi/10^8,\frac{1}{10^5}\log 2\}$.
\end{thm}

\begin{proof}
Let $\xi>0$. Assume that $\q G$ has $\xi$-uniform uniform exponential growth. Let $U\subset G$ be a finite symmetric subset containing the identity and denote $\Gamma=\group{U}$. Then $\Gamma$ falls exactly in one of the following two cases.

\textbf{Case 1.} \emph{$\Gamma$ is elementary.} If $\Gamma$ is loxodromic, it is virtually nilpotent. If $\Gamma$ is elliptic, then $\Gamma$ and $\q \Gamma$ are isomorphic (Lemma \ref{lem:quotient-elliptic}). Since $\q G$ has $\xi$-uniform uniform exponential growth, the subgroup $\q \Gamma$ is either virtually nilpotent or has $\xi$-uniform exponential growth. Thus $\Gamma$ is either virtually nilpotent or has $\xi$-uniform exponential growth.

\textbf{Case 2.} \emph{$\Gamma$ is non-elementary.} If $L(U)>10^4\delta_0$, then $\omega(U) \ge \frac{1}{10^5} \log 2$ by Lemma \ref{lem:breuillard-fujiwara}. 
	If $L(U)\le 10^4\delta_0$, then 
 by Lemma \ref{lem:small-energy}, $\omega(\q U)>\frac{1}{10^8}\omega(U)$. As $\Gamma$ is non-elementary, $\omega(U)>0$. In particular $\q \Gamma$ is not virtually nilpotent. Since $\q G$ has $\xi$-uniform uniform exponential growth, $\omega(\q U) \ge \xi$. By definition, $\omega(U)\ge \omega(\q U)\ge \xi$. 
 
 This yields that $\omega(U) \geqslant \min \{\frac{1}{10^8} \xi,\frac{1}{10^5} \log 2\}$, as claimed. 
\end{proof}

\begin{proof}[Proof of Theorem \ref{IT:main}] Let $\xi>0$, $N>0$, $\delta >0$ and $\kappa \geqslant \delta$. We assume that $G$ acts $(\kappa,N)$-acylindrically on a $\delta$-hyperbolic space $X$. We let $\lambda_0=  \frac{\Delta_0}{100\pi\sinh\rho}$  and $ \mu_0=\frac{100\pi \sinh \rho}{\delta_0}\cdot \frac{\kappa}{\delta}.$ By the choice of $\rho$ in \eqref{E:choice-of-rho} above, $\lambda_0$ depends on $1/N$, and $\mu_0$ on $N$, $\kappa$ and $\delta$, as claimed.  Let $\lambda\le \lambda_0$  and $ \mu\ge \mu_0.$ 
 After rescaling $X$, see Remark \ref{R:rescaling}, Theorems \ref{IT:main2-body} and \ref{IT:main3-body} apply. This yields the assertions of Theorem \ref{IT:main}. 
\end{proof}

\begin{proof}[Proof of Theorem \ref{IT:essence}] Let $\lambda >0$ and $\mu >0$. Let $N>0$, $\delta>0$ and $\kappa\geqslant \delta$. Let $G$ be a group that acts $(\kappa,N)$-acylindrically and non-elementarily on a $\delta$-hyperbolic space. By Theorem \ref{IT:main}, it remains to argue that $G$ admits a $C''(\lambda,\mu)$-small cancellation quotient. This works as follows. Let $h\in G$ be a loxodromic element. Such an element exists as the action of $G$ on $X$ is not elementary. Up to passing to a power of $h$, we may assume that  $||h||^{\infty}>10^3\delta$.  Moreover, there is a positive number $n>0$ such that for all $g\in E(h)$ we have that $h^{n}=gh^{n}g^{-1}$ or $h^{n}=gh^{- n}g^{-1}$, see \cite[Lemma 19]{ivanov_hyperbolic_1996}, \cite[Lemma 3.38]{coulon_partial_2016}. In particular, $\langle
h^n\rangle$  is normal in the maximal loxodromic subgroup $E(h)$ containing $h$. Let $k>0$. By Lemmas \ref{L: axis} and \ref{lem:acylindricity-intersection-axis}, we have that $\Delta (h^{kn})\leqslant (N+2)||h||^{\infty} +\kappa +250\delta$. On the other hand, $||h^{kn}||^{\infty}>k||h||^{\infty}$. Thus there is an exponent $k>0$ such that $\{g h^{kn} g^{-1}\mid g\in G\}$ satisfies the $C''(\lambda,\mu)$-condition. 
\end{proof}

\begin{rem}\label{R:product-set-growth} Recall that a group has \emph{product set growth} (for its symmetric subsets) if there is $a>0$ such that for all finite symmetric subsets $U$ that are not in a virtually nilpotent subgroup, we have that $\omega(U)\ge a \log(|U|)$. If $G$ has product set growth then $\overline{G}$ has product set growth, and vice versa. This is proved by replacing Lemma~\ref{lem:breuillard-fujiwara} by Theorem \ref{thm:growth-trichotomy-intro} in the proof of Theorem~\ref{IT:main2-body}, Theorem \ref{IT:main3-body} respectively. Indeed, let $\overline U\subset \overline{G}$ such that  $L(\overline{U})>10^4 \bar{\delta}$ be a non-empty symmetric set. Then $\overline{U}^2$ contains a loxodromic isometry. Thus there is $n>0$, that does not depend on the choice of $\overline{U}$, such that $L(\overline U^n)>10^4\overline{\kappa}$. Therefore Theorem \ref{thm:growth-trichotomy-intro} yields the required estimates in this case.

  We note that the constant of Theorem \ref{thm:growth-trichotomy-intro} can be explicitly given and depends only on $N$. Also, the number $n$ can be made explicit and depends on $\overline \delta$ and $\overline{\kappa}$. Combined with Remark \ref{rem:small-cancellation-acylindricity}  and Remark \ref{R:rescaling} this yields Remark \ref{IR:product-set-growth}. 
\end{rem}

\section{Common quotients}
Finally, we explain the proof of Corollary \ref{IC:common-quotients}. 
 We start with the following consequence of \cite{coulon_product_2022}. 
 
 \begin{prop}\label{P: groups-large-torsion-balls}
 There  are $\xi>0$, $N>0$, $\kappa>0$, $\delta>0$ such that the following holds. For every $i>0$, there is a $2$-generated group $H_i$ and a set $S_i:=\{s_{i_1},s_{i_2}\}$ such that  
 \begin{enumerate}
 \item the set  $S_i:=\{s_{i_1},s_{i_2}\}$ generates $H_i$,
 \item  all elements in $(S_i\cup S_{i}^{-1})^i$ are torsion elements,
\item there is a $\delta$-hyperbolic geodesic metric space $X_i$ such that $H_i$ acts $(\kappa,N)$-acylindrically on $X_i$,
\item the action of $H_i$ on $X_i$ is not elementary, 
\item the group $H_i$ has $\xi$-uniform uniform exponential growth,
\item every loxodromic subgroup of $H_i$ is cyclic.  
 \end{enumerate}
 \end{prop}

 \begin{proof}[Sketch] We explain how this result follows from \cite{coulon_product_2022} and Lemma \ref{lem:small-energy}. Let $G$ be a (2-generated) non-cyclic torsion-free hyperbolic group. We let $n>0$ and denote by $G/G^n$ the quotient of $G$ by the normal subgroup $\langle g^n\mid  g\in G\rangle$. There is $\xi >0$ such that for sufficiently large odd exponent $n$, the quotient $G/G^n$ has $\xi$-uniform uniform exponential growth \cite[Corollary 1.3]{coulon_product_2022}. More precisely, if $V\subset G/G^n$ is a finite symmetric subset that does not generate a finite subgroup, then $\omega(V)>\xi$. 
 
 Moreover, Section 10.2 of \cite{coulon_product_2022} gives the following. There are $n_1>0$, $L_0>0$, $\delta>0$, $\kappa>0$ and $N>0$ such that for all odd exponents $n>n_1$ the following holds. The quotient $G/G^n$ is a direct limit of hyperbolic groups $G_i$ such that, for all $i\geqslant 0$: 
 \begin{enumerate}
 \item there are $\delta$-hyperbolic geodesic metric spaces $X_i$ such that $G_i$ acts properly and cocompactly on $X_i$. Moreover, this action is $(\kappa,N)$-acylindrical and non-elementary. 
 \item let $P_i$ be the set of hyperbolic primitive elements $h_i\in G_i$ such that $||h_i||\leqslant L_0\delta$. Then $G_{i+1}$ is the quotient of $G_i$ by the normal closure of $\langle h_i^n\mid h_i\in P_i\rangle$. 
 \item every loxodromic subgroup of $G_i$ is cyclic.  
 \item there is a $1$-Lipschitz map from $X_i$ to $X_{i+1}$. 
 \end{enumerate}
The point is that $L_0$, $\kappa$, $N$ and $\delta$ do neither depend on the index $i\geqslant 0$, nor on the choice of the exponent $n>n_1$.  
 
We now argue that for sufficiently large odd exponent $n>n_1$, each of the groups $G_i$ has $\min\{\xi,10^{-5}\log 2\}$-uniform uniform exponential growth. To this point, we want to apply Lemma \ref{lem:small-energy}. This is justified as follows: fix $\delta_0$, $\Delta_0$ and $\rho$ as in Section \ref{sec:growth-small}. Let $\lambda\le \frac{\Delta_0}{100\pi\sinh\rho}$  and $ \mu\ge \frac{100\pi \sinh \rho}{\delta_0}\cdot \frac{\kappa}{\delta}.$ By a standard argument, there is a number $n_2$, that depends on $\delta$, $\kappa$, $N$, $L_0$, $\lambda$ and $\mu$ such that for all $n>n_2$, the set of relators $\{ h_i^n \mid h_i\in P_i\}\subset G_i$ satisfies the $C''(\lambda,\mu)$-small cancellation condition for the action of $G_i$ on $X_i$. See for instance \cite[Proposition 6.29 and Remark 6.30]{dahmani_hyperbolically_2017}. 
By construction, the exponent $n_2$ does not depend on the index $i$. We now fix an odd number $n>n_2$. 

After rescaling the spaces $X_i$, see Remark \ref{R:rescaling}, Lemma \ref{lem:small-energy} applies for each $G_i$ and its small cancellation quotient $G_{i+1}$. Note that the rescaling constant given by Remark \ref{R:rescaling} does not depend on the index $i$. 
  Let $U_i\subset G_i$ be a symmetric subset and let $\Gamma_i=\langle U_i \rangle$. If $\Gamma_i$ is elementary, then it is cyclic or finite. Thus we may assume that $\Gamma_i$ is not elementary. Then $\Gamma_i$ has exponential growth and $\omega(U_i)>0$.  If $L(U_i)>10^4\delta_0$, then $\omega(U_i)>10^{-5}\log 2$ by Lemma \ref{lem:breuillard-fujiwara}. Otherwise Lemma \ref{lem:small-energy} implies that the image of $\Gamma_i$ in $G_{i+1}$ has exponential growth. In particular, the image of $\Gamma_i$ in $G_{i+1}$ is not elementary. As the map from $X_j$ to $X_{j+1}$ is $1$-Lipschitz, the $\ell^{\infty}$-energy of each of the images of $U_i$ in $G_j$, $j\geqslant i$, is bounded by $10^4\delta_0$. This means that Lemma \ref{lem:small-energy} applies to the image of $U_i$ in $G_j$, for all $j>i$. Thus, by induction, the image of $\Gamma_i$ in any of the groups $G_j$, $j>i$, is not elementary, hence, it is not finite. Thus the image of $\Gamma_i$ in $G/G^n$ is not finite. As $G/G^n$ has $\xi$-uniform uniform exponential growth it follows that $\omega(U_i)\geqslant \xi$. This implies the claim. 
  
    We now take $H_i$ to be an appropriate subsequence of the groups $G_i$. 
 \end{proof}
 
 \begin{rem} Lemma \ref{lem:small-energy} is not necessarily needed to prove Proposition \ref{P: groups-large-torsion-balls}. In fact, it is straight-forward to adapt the argument of Section 10.3 of \cite{coulon_product_2022}  to prove this result.  
 \end{rem}

For the remainder of this section, let $(H_i)$ be a family of $2$-generated groups  such that the group $H_i$ is generated by $S_i:=\{s_{i_1},s_{i_2}\}$ and such that all elements in $(S_i\cup S_{i}^{-1})^i$ are torsion elements. By Proposition \ref{P: groups-large-torsion-balls}, we assume the following. Let $\xi>0$, $N>0$, $\kappa>0$, $\delta>0$ and $\delta$-hyperbolic geodesic metric spaces $X_i$ such that, for all $i>0$, 
\begin{enumerate}
\item the group $H_i$ acts $(\kappa,N)$-acylindrically and non-elementarily on $X_i$, 
\item the group $H_i$ has $\xi$-uniform uniform exponential growth,
\item every loxodromic subgroup of $H_i$ is cyclic.  
\end{enumerate}

 \begin{rem}\label{R:energy-main-application} The energy $L((S_i\cup S_i^{-1})^i)\le 200 \delta$. Indeed, if $L((S_i\cup S_i^{-1})^{i/2})>100 \delta$, then $(S_i\cup S_i^{-1})^i$ contains a loxodromic isometry by Proposition \ref{P:koubi}. This contradicts the choice of $S_i$. Thus $L((S_i\cup S_i^{-1})^{i/2})\le 100 \delta$ and the triangle inequality implies the claim. 
\end{rem}

%
%

Let $\langle t_1\rangle$ and $\langle t_2\rangle$ be two infinite cyclic groups on generators $t_1$ and $t_2$, respectively. Let 
$$G=\langle t_1 \rangle * \langle t_2\rangle * H_3*H_4* \cdots * H_i \cdots .$$
Then $G$ has $\xi$-uniform uniform exponential growth. 
Moreover, there is a $\delta$-hyperbolic space $X$ on which $G$ acts $(\kappa',N)$-acylindrically, where $\kappa'=\kappa+200\delta$. This space is a tree of the spaces $X_i$ that is constructed as follows. 
We let $X_0=\{x_0\}$ and $X_1=\{x_1\}$ be a space consisting of one point, and define $H_1=\langle t_1 \rangle$ and $H_2=\langle t_2 \rangle$. For $i>2$, let $x_i\in X_i$ be a point almost minimising the energy of $S_i^i$. We denote by $(X_i,x_i)$ the space $X_i$ with the fixed base point $x_i$. Let $T$ be the Bass-Serre tree of $G$. For each coset (i.e. vertex of $T$) $gH_i$ we take a copy of $(X_i,x_i)$ that we denote by $(X_i^g,x_i^g)$. We build $X$ from the disjoint union of the spaces $(X_i^g,x_i^g)$ by adding a segment $(hx_{i}^g,x_j^{gh})$ if and only if  $(gH_i,ghH_j)$ is an edge of $T$. We furthermore assume that the segments $(hx_{i}^g,x_j^{gh})$ in $X$ have length $100\delta$.  We  observe that $X$ is a $\delta$-hyperbolic metric space, that $G$ acts by isometries on $X$ and  that there is a $G$-equivariant projection from $X$ onto $T$. Moreover, the action is indeed $(\kappa+200\delta,N)$-acylindrical. 

\begin{rem}\label{R:independent} Let $i>2$. Then there are two independent loxodromic elements $h_{i_1},h_{i_2}$ of $H_i$ such that $|h_{i_j}x_i-x_i| < 10^3\delta$. In particular, $\|h_{i_j}\|\leqslant 10^3\delta$. Indeed, recall that $L(S_i\cup S_i^{-1})\leqslant 100 \delta$. By Lemma \ref{L:power-large-energy} there is $n_i>0$ such that $100\delta < L\left((S_i\cup S_i^{-1})^{n_i}\right) \leqslant 200 \delta$. Thus there is a loxodromic element $h_{i_1}$ in $(S_i\cup S_i^{-1})^{2n_i}$ \cite[Proposition 3.2]{koubi_croissance_1998}. Moreover, 
$$ 400\delta \geqslant L\left((S_i\cup S_i^{-1})^{2n_i}\right)\geqslant |h_{i_j}x_i-x_i|-\delta .$$
As $S_i$ generates $H_i$ there is $s_i\in S_i$ such that $h_{i_2}=s_ih_{i_1}s_i^{-1}$ and $h_{i_1}$ do not generate an elementary subgroup. In other words, $h_{i_1}$ and ${h_{i_2}}$ are independent. 
\end{rem}

\begin{lem}\label{L:common-quotient-injection} There are $\lambda_1>0$ and $\mu_1>0$ with the following property.  Let $\lambda \leqslant \lambda_1$ and let $\mu\geqslant \mu_1$. Let $\overline G$ be a $C''(\lambda,\mu)$-small cancellation quotient of $G$. Then the following holds. 
\begin{enumerate}
\item For all $i\geqslant 3$, the set $(S_i\cup S_i^{-1})^i\subset H_i$ injects into $\overline G$.
\item The quotient $\overline G$ has exponential growth.
\item There is an acylindrical and non-elementary action of $\overline G$ on a hyperbolic space. 
\end{enumerate}  
\end{lem}
\begin{proof} Recall that $\delta_0$, $\Delta_0$, $\rho_0$ and $\overline \delta$ are the constants of the Small Cancellation Theorem. 
Fix $\delta_0$, $\Delta_0$, $\rho_0$ and $\rho$ as in Section \ref{sec:growth-small}. Note that $\rho \geqslant \max\{\rho_0,10^4\delta_0\}$, see \eqref{E:choice-of-rho}. We let $\lambda_1=\frac{\Delta_0}{100\pi\sinh\rho}$ and $\mu_1= \frac{100\pi \sinh \rho}{\delta_0}\cdot \frac{\kappa}{\delta}$, and $\lambda\le \lambda_1$  and $ \mu\ge \mu_1.$ Let $\mathcal Q$ be a loxodromic moving family satisfying the $C''(\lambda,\mu)$-small cancellation condition. After rescaling $X$ if necessary, see Remark \ref{R:rescaling}, we may assume that $G$, $X$ and $\mathcal Q$ satisfy the hypothesis of the Small Cancellation Theorem (Lemma \ref{lem:small-cancellation-theorem}). In particular, we assume that $\delta\leqslant \delta_0$, and let $\overline X$ be given by  Lemma \ref{lem:small-cancellation-theorem}. The action of $\overline G$ on $\overline X$ is acylindrical by Lemma \ref{lem:small-cancellation-acylindricity}.

We prove (1). By Remark \ref{R:energy-main-application} and construction, $L((S_i\cup S_i^{-1})^i)\le 200 \delta_0$ for the action of $G$ on $X$, for all $i>2$. Let $p_i\in X$ be a point almost minimising the energy of $(S_i\cup S_i^{-1})^i$. Let $s_1\not=s_2\in (S_i\cup S_i^{-1})^i$ and let $g=s_1^{-1}s_2$. Let $Y$ be the invariant cylinder of a loxodromic subgroup, and let $y_0$ be a projection of $p_i$ on $Y$, and $y_1$ a projection of $gp_i$, respectively. By quasi-convexity of $Y$, see Lemma \ref{L: invariant cylinder}, $|p_i-gp_i|\geqslant |y_0-y_1|- 34\delta$, see \cite[Chapitre 10, Proposition 2.1]{coornaert_geometrie_1990}. By the triangle inequality $|gp_i-p_i|\leqslant 402\delta_0$, hence, $|y_0-y_1|\leqslant 436\delta_0$. Now the Greendlinger lemma, Lemma \ref{lem:greendlingers-lemma}, implies that $\overline g\not =1$, hence, $\overline {s_1}\not=\overline {s_2}$.  


Moreover, we may apply Lemma \ref{lem:small-energy} to the set $U=S_i\cup S_i^{-1}$. As the group $H_i$ has exponential growth, $\omega(S_i\cup S_i^{-1})>0$. We conclude that $\overline G$ has exponential growth. This is assertion (2).  

Finally, we prove (3). Assertion (2) implies that $\overline G$ is not virtually nilpotent. In particular, the action of $\overline G$ on $\overline X$ is not loxodromic. If the action of $\overline G$ is elliptic, by Lemma \ref{lem:lifting-elliptic-subgroups} and Lemma \ref{lem:small-cancellation-theorem}(iii), then $\overline G$ is either the isomorphic image of an elliptic subgroup of $G$, or it is in $\stab Y /H$, for some $(H,Y)\in \mathcal Q$. This is a contradiction. Indeed, the elliptic subgroups of $G$ are exactly the elliptic subgroups of the groups $H_i$, $i\geqslant 3$, which are finite. On the other hand, $\mathcal Q$ is a loxodromic moving family. This means that $H$ is virtually cyclic, hence, that $\stab Y/H$ is finite for all $(H,Y)\in \mathcal Q$. We conclude that the action of $\overline G$ on $\overline X$ is not elementary. 
\end{proof}

Next we define a set of relators $R$. 
We let $\{h_{i_1},h_{i_2}\}$ be two independent loxodromic elements of $H_i$, $i>2$ such that $\|h_{i_j}\|\leqslant |h_{i_j}x_i-x_i| < 10^3\delta$. See Remark \ref{R:independent}. 

 Let $\mu>0$ and let $\alpha>\mu ,\beta_i>0$ and $\gamma$ be positive numbers. We then define $R_1$  to be the set of the following elements of $G$: for each $i>2$ the set $R_1$ contains the elements 
\begin{align*}
s_{i_1}t_1t_2^{\alpha  20(i-1) +1}t_1t_2^{\alpha 20(i-1) +2}\cdots t_1t_2^{\alpha 20i}\\
s_{i_2}t_2t_1^{\alpha 20(i-1) +1}t_2t_1^{\alpha 20(i-1) +2}\cdots t_2t_1^{\alpha 20i}.
\end{align*}
The set $R_2$ contains, for each $i>2$, the elements
\begin{align*}
t_{1}h_{i_1}^{\gamma}h_{i_2}^{\beta_1\gamma}h_{i_1}^{\gamma}h_{i_2}^{\beta_2\gamma}\cdots h_{i_1}^{\gamma}h_{i_2}^{\beta_\mu \gamma} \\
t_{2}h_{i_1}^{\beta_1\gamma}h_{i_2}^{\gamma}h_{i_1}^{\beta_2\gamma}h_{i_2}^{\gamma}\cdots h_{i_1}^{\beta_\mu \gamma}h_{i_2}^{\gamma}.
\end{align*}
We let $R=R_1\cup R_2$ and fix $0<\lambda<\max\{\lambda_0,\lambda_1\}$ and $\mu>\{\mu_0,\mu_1\}$, where $\lambda_0$, $\mu_0$ and $\lambda_1,\mu_1$ are the constants given by Theorem \ref{IT:main} and Lemma \ref{L:common-quotient-injection}, respectively. As usual, there are $\mu$, $\alpha,\beta_i$ and $\gamma>0$ such that $R$ satisfies the $C''(\lambda,\mu)$-small cancellation condition for the action of $G$ on $X$, and we fix $R$ under this condition. 

\begin{rem} In a setting of a proper action on a hyperbolic space, the $C''(\lambda,\mu)$-condition implies that the relator set is finite up to conjugation. For instance, as previously mentioned, every classical $C''(\lambda)$-small cancellation group is finitely presented. In the above situation the action of $G$ on $X$ is not proper. This allows for the choice of $R$ under the $C''(\lambda,\mu)$-condition, even though $R$ is not finite up to conjugation. Another example is the action of a free product on its Bass-Serre tree. In this case the $C''(\lambda,\mu)$-condition corresponds to the classical $C_*''(\lambda)$-condition over the free product, and one can indeed make infinitely presented classical $C_*''(\lambda)$-groups over a free product. 
\end{rem}

 Finally, by Theorem \ref{IT:main}, the group $\overline{G}=G/\langle\langle R \rangle\rangle$ has uniform uniform exponential growth, and acts acylindrically on a hyperbolic space. By the choice of the set $R_1$, the group $\overline{G}$ is finitely generated.  By the choice of the set $R_2$, the group $\overline{G}$ is generated by each of the sets $S_i$. By Lemma \ref{L:common-quotient-injection}, for all $i\geqslant 3$, the ball $(S_i\cup S_i^{-1})^i\subset H_i$ injects into $\overline G$. Recall that, for all $i\geqslant 1$, all elements of $(S_i\cup S_i^{-1})^i$ are torsion elements. In particular, all elements of $(S_i\cup S_i^{-1})^i$ are elliptic.  This implies that $\q G$ does not have the short loxodromic property.
 
 Also by Lemma \ref{L:common-quotient-injection}, $\overline G$ has exponential growth, hence, it has uniform exponential growth. Moreover, the action of $\overline G$ on $\overline X$ is not elementary. This  concludes the proof of Corollary~\ref{IC:common-quotients}.

\bibliographystyle{alphaurl}
\bibliography{biblio}

\end{document}